\documentclass[10 pt, draft]{article}
\title{Abelian covers of surfaces and the homology of the level $L$ mapping class group}
\author{Andrew Putman\footnote{Supported in part by NSF grant DMS-1005318}}
\usepackage{amsmath}
\usepackage{amssymb}
\usepackage{amsthm}
\usepackage{epsfig}
\usepackage[margin=1.25in]{geometry}
\usepackage{amsfonts}
\usepackage[font=small,format=plain,labelfont=bf,up,textfont=it,up]{caption}
\usepackage{amscd}
\usepackage{pinlabel}
\usepackage{stmaryrd}
\usepackage{type1cm}
\usepackage{calc}
\usepackage{pbox}
\usepackage{mathptmx}  

\theoremstyle{plain}
\newtheorem{theorem}{Theorem}[section]
\newtheorem{maintheorem}{Theorem}

\newtheorem{lemma}[theorem]{Lemma}

\newtheorem{claim}{Claim}
\newtheorem{step}{Step}

\newcommand\BeginClaims{\setcounter{claim}{0}}

\newcommand\BeginSteps{\setcounter{step}{0}}

\theoremstyle{definition}
\newtheorem*{definition}{Definition}

\theoremstyle{remark}
\newtheorem*{remark}{Remark}

\DeclareMathOperator{\Ker}{ker}
\DeclareMathOperator{\Image}{Im}

\DeclareMathOperator{\Mod}{Mod}

\DeclareMathOperator{\Sp}{Sp}

\DeclareMathOperator{\SL}{SL}

\newcommand\C{\ensuremath{\mathbb{C}}}
\newcommand\Z{\ensuremath{\mathbb{Z}}}
\newcommand\Q{\ensuremath{\mathbb{Q}}}

\DeclareMathOperator{\HH}{H}

\DeclareMathOperator{\Min}{min}

\newcommand\Span[1]{\ensuremath{\langle #1 \rangle}}
\newcommand\CaptionSpace{\hspace{0.2in}}
\DeclareMathOperator{\Dim}{dim}
\newcommand\Set[3][]{\ensuremath{\{\text{#2 $|$ \pbox[t]{\textwidth}{#3$\}$#1 }}}}

\newcommand\Figure[3]{
\begin{figure}[t]
\centering
\centerline{\psfig{file=#2,scale=60}}
\caption{#3}
\label{#1}
\end{figure}}

\newcommand\KRHO[1]{\ensuremath{\llbracket #1 \rrbracket}}
\newcommand\CAB[1]{\ensuremath{\overline{#1}}}
\newcommand\CL[2]{\ensuremath{\langle\!\!\langle #1 \rangle\!\!\rangle^{#2}}}
\newcommand\CLL[1]{\ensuremath{\langle\!\!\langle #1 \rangle\!\!\rangle}}
\newcommand\BBB[1]{\ensuremath{\mathcal{B}_{#1}}}

\newcommand\KKK[1]{\ensuremath{K_{#1}}}
\newcommand\CCC[1]{\ensuremath{C_{#1}}}
\newcommand\QQQ[1]{\ensuremath{Q_{#1}}}
\newcommand\SSS{\ensuremath{\mathcal{S}_g}}
\newcommand\SSSS[1]{\ensuremath{\mathcal{S}_g(#1)}}
\newcommand\SKER{\ensuremath{\mathcal{K}_g}}
\newcommand\III[1]{\ensuremath{I_{#1}}}
\newcommand\CVR[1]{\ensuremath{S^K_{#1}}}
\newcommand\REID[2]{\ensuremath{\omega(#1,#2)}}

\begin{document}

\maketitle

\begin{abstract}
We calculate the first homology group of the mapping class group with coefficients in the first rational homology group
of the universal abelian $\Z / L$-cover of the surface.  If the surface has one marked point, then the answer
is $\Q^{\tau(L)}$, where $\tau(L)$ is the number of positive divisors of $L$.  If the surface instead has
one boundary component, then the answer is $\Q$.  We also perform the same calculation for the level $L$ subgroup
of the mapping class group.  Set $H_L = \HH_1(\Sigma_g;\Z/L)$.
If the surface has one marked point, then the answer is $\Q[H_L]$, the rational group ring of $H_L$.  
If the surface instead has one boundary component, then the answer is $\Q$.
\end{abstract}

\section{Introduction}

Let $\Sigma_{g,b}^n$ be an oriented genus $g$ surface with $b$ boundary components and $n$ marked points and let
$\Mod_{g,b}^n$ be its {\em mapping class group}.  This is the group of homotopy classes of orientation-preserving
diffeomorphisms of $\Sigma_{g,b}^n$ that act as the identity on the boundary components and marked points.  We
will usually omit the $b$ and the $n$ if they vanish.  The homology groups of $\Mod_{g,b}^n$, which play important
roles in both algebraic geometry and low-dimensional topology, have been studied intensely for the past 40 years.  The
culmination of much recent work is the resolution of the Mumford conjecture by Madsen and Weiss \cite{MadsenWeiss},
which identifies $\HH^{\ast}(\Mod_{g,b}^n;\Q)$ in a stable range.

\paragraph{Twisted coefficient systems.}
For many applications, it is important to know the homology groups of $\Mod_{g,b}^n$ with respect to various twisted 
coefficient systems.  For simplicity, assume that $(b,n) \in \{(0,0),(1,0),(0,1)\}$.  
A lot is known about coefficient systems that factor through the standard {\em symplectic
representation} of $\Mod_{g,b}^n$.  This is the natural representation $\Mod_{g,b}^n \rightarrow \Sp_{2g}(\Z)$ that arises 
from the action of $\Mod_{g,b}^n$ on $\HH_1(\Sigma_{g,b}^n;\Z)$.  Its target 
is the symplectic group because the action preserves the algebraic intersection form.  For any rational
representation $V$ of the algebraic group $\Sp_{2g}$, Looijenga \cite{LooijengaTwisted} has completely determined 
$\HH^{\ast}(\Mod_{g};V)$ as a module over $\HH^{\ast}(\Mod_{g};\Q)$ in a stable range.  Over $\Z$, a bit less is known.
Morita \cite{MoritaFamilies} has calculated $\HH_1(\Mod_{g,b}^n;\HH_1(\Sigma_{g,b}^n;\Z))$ for $g \geq 3$.  For $b \geq 1$,
this was later generalized by Kawazumi \cite{Kawazumi}, who calculated 
$\HH^{\ast}(\Mod_{g,b}^n;(\HH^1(\Sigma_{g,b}^n;\Z))^{\otimes k})$ as a module over $\HH^{\ast}(\Mod_{g,b}^n;\Z)$
in a stable range.

Fix some $L \geq 2$.  
In this paper, we calculate the first homology group of the mapping class group with coefficients in the first rational
homology group of the universal abelian $\Z/L$-cover of the surface (see below for the definition).
We remark that this representation does {\em not}
factor through $\Sp_{2g}(\Z)$.  
Our techniques also give results for certain finite-index
subgroups of the mapping class group.  These results play an important technical role in a
recent pair of papers by the author \cite{PutmanSecondHomologyLevel, PutmanPicard}
that study the second cohomology group and Picard group
of the moduli space of curves with level $L$ structures.

\paragraph{Universal abelian $\Z/L$-cover.}
Let $\KKK{g}$
be the kernel of the natural map $\pi_1(\Sigma_g) \rightarrow \HH_1(\Sigma_g;\Z/L)$.  The group $\KKK{g}$ is the
fundamental group of the universal abelian $\Z/L$-cover of $\Sigma_g$.  Since $\Mod_{g}^1$ fixes a basepoint
on $\Sigma_g$, it acts on $\pi_1(\Sigma_g)$.  This action preserves $\KKK{g}$.  We thus obtain an action
of $\Mod_{g}^1$ on $\HH_1(\KKK{g};\Q)$, the first homology group of the universal abelian $\Z/L$-cover of $\Sigma_g$.  
This representation has been previously studied by Looijenga \cite{LooijengaPrym}, who essentially determined its image.

\begin{remark}
In \cite{LooijengaPrym}, Looijenga more generally studied the actions of appropriate finite-index subgroups of $\Mod_{g}^1$
on the first rational homology groups $V$ of arbitrary finite abelian covers of $\Sigma_{g}$.  Letting $\Mod_{g}^1(L)$
denote the level $L$ subgroup of $\Mod_g^1$ (see below), we can choose $L$ so that $\Mod_{g}^1(L)$ acts on $V$.  It
then follows from Lemma \ref{lemma:decomposecover} below that $V$ appears a direct summand in the $\Mod_{g}^1(L)$-module
$\HH_1(\KKK{g};\Q)$, so one can use our results to study $V$ as well.
\end{remark}

\paragraph{Statements of theorems.}
Let $\tau(L)$ be the number of positive divisors of $L$ (including $1$ and $L$).  Our first theorem
is as follows.

\begin{maintheorem}
\label{theorem:mainmcg}
For $g \geq 4$ and $L \geq 2$, we have $\HH_1(\Mod_{g}^1;\HH_1(\KKK{g};\Q)) \cong \Q^{\tau(L)}$.
\end{maintheorem}

In fact, our proof of Theorem \ref{theorem:mainmcg} also gives a result for the 
{\em level $L$ subgroup} of $\Mod_g^1$, denoted $\Mod_{g}^1(L)$.  This is the kernel of the action of $\Mod_{g}^1$ on
$\HH_1(\Sigma_{g}^1;\Z/L)$.  Far less is known about its homology.  The only previous result of which
the author is aware is a paper of Hain \cite{HainTorelli} that calculates $\HH_1(\Mod_{g,b}^n(L);V)$ for
rational representations $V$ of the algebraic group $\Sp_{2g}$.  Our theorem is as follows.

\begin{maintheorem}
\label{theorem:mainclosed}
For $g \geq 4$ and $L \geq 2$, we have
$\HH_1(\Mod_{g}^1(L);\HH_1(\KKK{g};\Q)) \cong \Q[H_L]$,
where $H_L$ equals $\HH_1(\Sigma_g;\Z/L)$ and $\Q[H_L]$ is the rational group ring 
of the abelian group $H_L$.
\end{maintheorem}

\begin{remark}
Both $\HH_1(\Mod_{g}^1(L);\HH_1(\KKK{g};\Q))$ and $\Q[H_L]$ possess natural $\Mod_{g}^1$-actions.  The
action of $\Mod_g^1$ on $\HH_1(\Mod_{g}^1(L);\HH_1(\KKK{g};\Q))$ comes from  conjugation, and the action on $\Q[H_L]$
factors through the symplectic group.  The isomorphism in Theorem \ref{theorem:mainclosed} is equivariant
with respect to these actions.
\end{remark}

Somewhat surprisingly, things are quite different for surfaces with boundary.  Define $\Mod_{g,1}(L)$ to be the
kernel of the action of $\Mod_{g,1}$ on $\HH_1(\Sigma_{g,1};\Z/L)$.  Fixing a basepoint
for $\pi_1(\Sigma_{g,1})$ on $\partial \Sigma_{g,1}$, the groups $\Mod_{g,1}$ and $\Mod_{g,1}(L)$ act 
on $\pi_1(\Sigma_{g,1})$.  Define $\KKK{g,1}$ to be the kernel of the map 
$\pi_1(\Sigma_{g,1}) \rightarrow \HH_1(\Sigma_{g,1};\Z/L)$.  The group $\KKK{g,1}$ is the fundamental
group of the universal abelian $\Z/L$-cover of $\Sigma_{g,1}$ and is preserved by the actions
of $\Mod_{g,1}$ and $\Mod_{g,1}(L)$.  We then have the following theorem.

\begin{maintheorem}
\label{theorem:mainboundary}
For $g \geq 4$ and $L \geq 2$, we have 
$$\HH_1(\Mod_{g,1};\HH_1(\KKK{g,1};\Q)) \cong \HH_1(\Mod_{g,1}(L);\HH_1(\KKK{g,1};\Q)) \cong \Q.$$
\end{maintheorem}

\begin{remark}
The group $\Mod_{g}$ does not act on the universal abelian $\Z/L$-cover of $\Sigma_g$.  Each individual
mapping class can be lifted to a diffeomorphism of the cover, but a fixed basepoint is necessary to make
this lift canonical and thereby provide a representation of the entire group.  The best one can achieve
is as follows.  There is a Birman exact sequence (see \S \ref{section:mcgprelim}) of the form
\begin{equation}
\label{eqn:birman}
1 \longrightarrow \pi_1(\Sigma_g) \longrightarrow \Mod_g^1 \longrightarrow \Mod_g \longrightarrow 1.
\end{equation}
Since $\Mod_{g}^1$ acts on $\HH_1(\KKK{g};\Q)$, the group $\Mod_g$ acts on the 
$(\HH_1(\KKK{g};\Q))_{\pi_1(\Sigma_g)}$.  However, using the transfer map (see Lemma \ref{lemma:supertransfer} 
below) one can show 
that this ring of coinvariants is simply $\HH_1(\Sigma_g;\Q)$, so no new representation is obtained.
\end{remark}

\paragraph{Comments on the proofs.}
The key observation underlying the proofs of our theorems is as follows.  The group $\Mod_{g}^1$ contains a natural
copy of $\pi_1(\Sigma_g)$, known as the ``point-pushing subgroup'' (see \S \ref{section:mcgprelim} below).  This
fits into the Birman exact sequence \eqref{eqn:birman} above.
While the action of $\Mod_{g}^1$ on $\KKK{g}$ is very complicated, the action of $\pi_1(\Sigma_g)$ on
$\KKK{g} \lhd \pi_1(\Sigma_g)$ is simply conjugation.  Moreover, it turns out that in some vague sense the action of
$\Mod_{g}^1$ on $\HH_1(\KKK{g};\Q)$ is ``concentrated'' in the action of $\pi_1(\Sigma_g)$ on $\HH_1(\KKK{g};\Q)$.  
Intuitively, this happens because (as noted in the remark above) the quotient of $\Mod_{g}^1$ 
by $\pi_1(\Sigma_g)$, namely $\Mod_g$, does {\em not} act in any natural way on $\KKK{g}$.  We prove our
theorems by carefully examining all of these groups and actions.

\paragraph{Some related results.}
Some additional related results should be mentioned.  First, Ivanov \cite{IvanovTwisted} has proven
a homological stability result for the homology of $\Mod_{g,1}$ with respect to very general systems
of coefficients (those of ``bounded degree''; the system $\HH_1(\KKK{g,1};\Z)$ satisfies this condition).  This 
generalizes Harer's \cite{HarerStability}
well-known untwisted homological stability theorem for the mapping class group.
Ivanov's theorem has been extended to $\Sigma_{g,b}$ for $b > 1$ by
Boldsen \cite{BoldsenStability}.  We remark that such a result is false for closed surfaces.  Indeed, in
\cite[Corollary 5.4]{MoritaFamilies}, Morita showed that
$$\HH_1(\Mod_g;\HH_1(\Sigma_g;\Z)) \cong \Z / (2g-2)\Z$$
for $g \geq 2$.  In a somewhat different direction, a recent series of papers by Anderson and
Villemoes \cite{AndersonVill1, AndersonVill2, AndersonVill3} calculate the first homology groups of $\Mod_{g,b}^n$
with coefficients in certain spaces of functions on representations varieties of $\Mod_{g,b}^n$.

\paragraph{Outline of paper.}
In \S \ref{section:preliminaries}, we discuss some background results about group cohomology and the mapping class
group.  Next, in \S \ref{section:cast} we introduce a number of groups and group actions that will play
important roles in our paper.  At the end of this section, we state two key technical lemmas whose proofs
are postponed until later.  In \S \ref{section:mainproofs}, we prove
our main theorems (assuming the truth of these two technical lemmas).  In \S \ref{section:mainlemma}, we give the outline
of the proof of our two key technical lemmas, reducing them to two other results, the first of which is proven in \S \ref{section:linearalgebra} (using
some preliminary calculations that are first done in \S \ref{section:generatorsandrelations})
and the second in \S \ref{section:kills}.

\paragraph{Notation and conventions.}
We will denote by $i(x,y) \in \Z/L$ the algebraic intersection number of $x,y \in \HH_1(\Sigma_g;\Z/L)$.
All surfaces we mention will contain a basepoint unless otherwise specified, 
and all maps between surfaces will respect this basepoint.
Also, if $G$ is a group, then we define $[g_1,g_2] = g_1 g_2 g_1^{-1} g_2^{-1}$ and $g_1^{g_2} = g_2 g_1 g_2^{-1}$ for $g_1, g_2 \in G$.

\paragraph{Acknowledgments.}
I wish to thank an anonymous referee who pointed out Lemma \ref{lemma:higherpairing} below to me.  This
dramatically simplified my original proofs.

\section{Preliminaries}
\label{section:preliminaries}

\subsection{Group homology}

We begin by reviewing some facts about group homology and establishing
some notation (see \cite{BrownCohomology} for more details).

\paragraph{Degree zero.}
Let $G$ be a group and $M$ be a $G$-module.  The {\em coinvariants}
of $M$, denoted $M_G$, is the quotient $M/K$, where $K$ is the submodule spanned by the set
$\{\text{$g \cdot x - x$ $|$ $x \in M$, $g \in G$}\}$.  We have $\HH_0(G;M) = M_G$.

\paragraph{The five-term exact sequence.}
Let
$$1 \longrightarrow K \longrightarrow G \longrightarrow Q \longrightarrow 1$$
be a short exact sequence of groups and let $M$ be a $G$-module.
We then have a 5-term exact sequence
$$\HH_2(G;M) \longrightarrow \HH_2(Q;M_K) \longrightarrow (\HH_1(K;M))_Q \longrightarrow \HH_1(G;M) \longrightarrow
\HH_1(Q;M_K) \longrightarrow 0.$$

\paragraph{The long exact sequence.}
Let $G$ be a group and let
$$0 \longrightarrow M_1 \longrightarrow M_2 \longrightarrow M_3 \longrightarrow 0$$
be a short exact sequence of $G$-modules.  Then there is a long exact sequence of the form
$$\cdots \longrightarrow \HH_k(G;M_1) \longrightarrow \HH_k(G;M_2) \longrightarrow \HH_k(G;M_3) \longrightarrow
\HH_{k-1}(G;M_1) \longrightarrow \cdots.$$

\paragraph{The transfer map.}
If $G_2 < G_1$ are groups satisfying $[G_1:G_2] < \infty$ and $M$ is a $G_1$-module, then for all $k$ there exists a
{\em transfer map} of the form $t : \HH_k(G_1;M) \rightarrow \HH_k(G_2;M)$
(see, e.g., \cite[Chapter III.9]{BrownCohomology}).
The key property of $t$ (see \cite[Proposition III.9.5]{BrownCohomology})
is that if $i : \HH_k(G_2;M) \rightarrow \HH_k(G_1;M)$ is the map induced
by the inclusion, then $i \circ t : \HH_k(G_1;M) \rightarrow \HH_k(G_1;M)$ is multiplication by  $[G_1:G_2]$.
In particular, if $M$ is a $G_1$-vector space over $\Q$,
then we obtain a right inverse $\frac{1}{[G_1:G_2]} t$ to $i$.  This yields the following standard lemma.

\begin{lemma}
\label{lemma:transfer}
Let $G_2 < G_1$ be groups satisfying $[G_1:G_2] < \infty$ and let $M$ be a $G_1$-vector space over $\Q$.  Then the
map $\HH_k(G_2;M) \rightarrow \HH_k(G_1;M)$ is surjective for all $k \geq 0$.
\end{lemma}

Assume now that $M = \Q$ and that $\Gamma$ is a group acting on $G_2$ and $G_1$ such that the inclusion is
$\Gamma$-equivariant.  The induced map $\HH_k(G_2;M) \rightarrow \HH_k(G_1;M)$ is therefore
$\Gamma$-equivariant.  Moreover, the map $t : \HH_k(G_1;M) \rightarrow \HH_k(G_2;M)$ is also $\Gamma$-equivariant, so
the surjection $\HH_k(G_2;M) \rightarrow \HH_k(G_1;M)$ splits in a $\Gamma$-equivariant manner.  We
obtain the following lemma.

\begin{lemma}
\label{lemma:decomposecover}
Fix $k \geq 0$, and let $G_2 < G_1$ be groups satisfying $[G_1:G_2] < \infty$.
Let $\Gamma$ be a group acting on $G_1$ and $G_2$ such
that the inclusion map $G_2 \rightarrow G_1$ is $\Gamma$-equivariant.  Define $C$ to be the kernel of the
surjection $\HH_k(G_2;\Q) \rightarrow \HH_k(G_1;\Q)$.  We then have a $\Gamma$-invariant splitting
$\HH_k(G_2;\Q) \cong \HH_k(G_1;\Q) \oplus C$.
\end{lemma}

Finally, for finite-index normal subgroups, the Hochschild-Serre spectral sequences implies the following
strengthening of Lemma \ref{lemma:transfer}.

\begin{lemma}
\label{lemma:supertransfer}
Let $G_2 \lhd G_1$ be groups satisfying $[G_1:G_2] < \infty$ and let $M$ be a $G_1$-vector space over $\Q$.  Then
$\HH_k(G_1;M) \cong (\HH_k(G_2;M))_{G_1}$ for all $k \geq 0$.
\end{lemma}

\subsection{Rational group rings}

Let $G$ be a finite group and let $\Q[G]$ be the rational group ring of $G$.  We will consider
$\Q[G]$ to be a left $G$-module.
Let $\epsilon : \Q[G] \rightarrow \Q$ be the {\em augmentation map}, i.e.\ the unique linear map
such that $\epsilon(g) = 1$ for all $g \in G$.  The map $\epsilon$ is a map of $G$-modules, where
$\Q$ has the trivial $G$-action.  Its kernel is the {\em augmentation ideal} $I(G)$.  We thus
have a short exact sequence of $G$-modules
\begin{equation}
\label{eqn:augseq}
0 \longrightarrow I(G) \longrightarrow \Q[G] \stackrel{\epsilon}{\longrightarrow} \Q \longrightarrow 0.
\end{equation}
Set $\theta = \sum_{g \in G} g \in \Q[G]$.
The element $\theta$ is invariant under $G$, and the exact sequence \eqref{eqn:augseq} splits
via the $G$-equivariant map $\psi : \Q \rightarrow \Q[G]$ defined by $\psi(1) = \frac{1}{|G|} \theta$.  The associated
projection $\phi : \Q[G] \rightarrow I(G)$ satisfies $\Ker(\phi) = \Span{\theta}$.  From these considerations,
we obtain the following lemma.

\begin{lemma}
\label{lemma:groupalgebra}
Let $G$ be a finite group.  Then $\Q[G] \cong \Q \oplus I(G)$, where $I(G)$ is isomorphic to the
quotient of $\Q[G]$ by $\Span{\theta}$.
\end{lemma}

\subsection{The mapping class group}
\label{section:mcgprelim}

\paragraph{Dehn twists.}
We will denote by $T_{\gamma}$ the left Dehn twist about a simple closed curve $\gamma$ on a surface.

\paragraph{The Birman exact sequence.}
For simplicity, assume that $g \geq 2$.
The Birman exact sequence describes the effect on the mapping class group of deleting a marked point
or gluing a disc to a boundary component.  The first version, due to Birman (see \cite{BirmanBook}), is
of the form
\begin{equation}
\label{eqn:birman1}
1 \longrightarrow \pi_1(\Sigma_{g,b}^p,\ast) \longrightarrow \Mod_{g,b}^{p+1} \longrightarrow \Mod_{g,b}^p
\longrightarrow 1.
\end{equation}
Here $\ast$ is a marked point and the map $\Mod_{g,b}^{p+1} \rightarrow \Mod_{g,b}^p$ comes from deleting $\ast$.
For $\gamma \in \pi_1(\Sigma_{g,b}^p,\ast)$, the associated mapping class in the kernel of \eqref{eqn:birman1}
``pushes'' the deleted marked point around the path $\gamma$.  For this reason, the kernel $\pi_1(\Sigma_{g,b}^p,\ast)$
is known as the ``point-pushing subgroup''.  If $\gamma \in \pi_1(\Sigma_{g,b}^p,\ast)$ can be realized by a simple
closed curve, then there is a nice formula for the associated ``point-pushing'' mapping class $\rho_{\gamma}$.  Namely, let
$\gamma_1$ and $\gamma_2$ be the boundary components of a regular neighborhood of $\gamma$ (see Figure \ref{figure:pointpushing}.c).
Assume that $\gamma_1$ lies to the left of $\gamma$ and $\gamma_2$ to the right.  
Then as is clear from Figure \ref{figure:pointpushing}.a--b,
we have $\rho_{\gamma} = T_{\gamma_1} T_{\gamma_2}^{-1}$.

\Figure{figure:pointpushing}{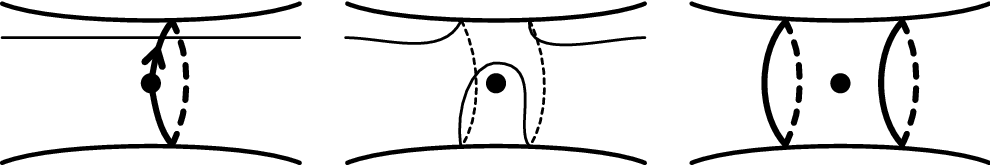}{a. $\gamma \in \pi_1(\Sigma_{g,b}^p,\ast)$ can be realized by a simple closed curve.
\CaptionSpace b. The effect of the ``point-pushing'' map $\rho_{\gamma}$.
\CaptionSpace c. $\rho_{\gamma} = T_{\gamma_1} T_{\gamma_2}^{-1}$}

The second form of the Birman exact sequence, due to Johnson \cite{JohnsonFinite}, is of the form
\begin{equation}
\label{eqn:birman2}
1 \longrightarrow \pi_1(U\Sigma_{g,b}^p) \longrightarrow \Mod_{g,b+1}^p \longrightarrow \Mod_{g,b}^p \longrightarrow 1.
\end{equation}
Here $U\Sigma_{g,b}^p$ is the unit tangent bundle of $\Sigma_{g,b}^p$ and the map
$\Mod_{g,b+1}^p \rightarrow \Mod_{g,b}^p$ comes from gluing a disc to a boundary component $\beta$ of
$\Sigma_{g,b+1}^p$ and extending mapping classes by the identity.  The fiber of the kernel $\pi_1(U\Sigma_{g,b}^p)$
corresponds to the mapping class $T_{\beta}$.  

The group $\Mod_{g,b+1}^p$ acts on $\HH_1(\Sigma_{g,b+1}^p;\Q)$.  Restrict this action to $\pi_1(U\Sigma_{g,b}^p) < \Mod_{g,b+1}^p$.
Since $T_{\beta}$ acts trivially on $\HH_1(\Sigma_{g,b+1}^p;\Q)$, the action of $\pi_1(U\Sigma_{g,b}^p)$ on
$\HH_1(\Sigma_{g,b+1}^p;\Q)$ factors through an action of $\pi_1(\Sigma_{g,b}^p,\ast)$.  This action has the following
simple description.

\begin{lemma}
\label{lemma:homologyaction}
The above action of $\pi_1(\Sigma_{g,b}^p,\ast)$ on $\HH_1(\Sigma_{g,b+1}^p;\Q)$ is given by the formula
$$\gamma(v) = v+i([\gamma],v) \cdot [\beta] \quad \quad (\text{$\gamma \in \pi_1(\Sigma_{g,b}^p,\ast)$, $v \in \HH_1(\Sigma_{g,b+1}^p;\Q)$}),$$
where the boundary component $\beta$ is oriented such that the interior of the surface is to its right.
\end{lemma}
\begin{proof}
It is enough to check this for elements $\gamma \in \pi_1(\Sigma_{g,b}^p,\ast)$ that can be realized by simple closed curves.  For such
curves, this formula is immediate from Figure \ref{figure:pointpushing}.a--b.
\end{proof}

\paragraph{The level $L$ subgroup.}
Now assume that $b=p=0$ and fix some $L \geq 2$.  The kernels of \eqref{eqn:birman1} and \eqref{eqn:birman2} both
lie in the level $L$ subgroup of the mapping class group.  We thus have short exact sequences
$$1 \longrightarrow \pi_1(\Sigma_g) \longrightarrow \Mod_{g}^1(L) \longrightarrow \Mod_{g}(L) \longrightarrow 1$$
and
$$1 \longrightarrow \pi_1(U\Sigma_g) \longrightarrow \Mod_{g,1}(L) \longrightarrow \Mod_{g}(L) \longrightarrow 1.$$
We will also refer to these as Birman exact sequences.

We will need two cohomological results about the level $L$ subgroups of the mapping class group, both
of which are due to Hain.  To simplify their statements, we will denote the whole mapping class group by
$\Mod_{g,1}(1)$ and $\Mod_{g}^1(1)$.

\begin{theorem}[{Hain, \cite{HainTorelli}}]
\label{theorem:h1modl}
For $g \geq 3$ and $L \geq 1$, we have $\HH_1(\Mod_{g,1}(L);\Q) = 0$.
\end{theorem}

\begin{theorem}[{Hain, \cite{HainTorelli}}]
\label{theorem:h1modlh}
For $g \geq 3$ and $L \geq 1$, we have
$$\HH_1(\Mod_{g,1}(L);\HH_1(\Sigma_{g};\Q)) \cong \HH_1(\Mod_{g}^1(L);\HH_1(\Sigma_g;\Q)) \cong \Q.$$
\end{theorem}

\begin{remark}
In fact, in \cite{HainTorelli} Hain calculated $\HH_1(\Mod_{g,b}^p(L);M)$ for all rational representations $M$ of the algebraic
group $\Sp_{2g}$.
\end{remark}

\section{The cast of characters}
\label{section:cast}

\subsection{Homology groups of abelian covers}

As in the introduction, define
$$\KKK{g} = \Ker(\pi_1(\Sigma_{g}) \rightarrow \HH_1(\Sigma_g;\Z/L)) \quad \text{and} \quad \KKK{g,1} = \Ker(\pi_1(\Sigma_{g,
1}) \rightarrow \HH_1(\Sigma_{g,1};\Z/L)).$$
Here the basepoint for $\pi_1(\Sigma_{g,1})$ lies on $\partial \Sigma_{g,1}$.  Also, let $\CVR{g}$ (resp.\
$\CVR{g,1}$) be the cover of $\Sigma_{g}$ (resp.\ $\Sigma_{g,1}$) corresponding to $\KKK{g}$ (resp.\
$\KKK{g,1}$).  We will identify $\HH_1(\KKK{g};\Q)$ and $\HH_1(\KKK{g,1};\Q)$ with $\HH_1(\CVR{g};\Q)$
and $\HH_1(\CVR{g,1};\Q)$, respectively.

We have actions of $\Mod_{g}^1$ and $\Mod_{g,1}$ on $\KKK{g}$ and $\KKK{g,1}$, respectively.  Define
$$\CCC{g} = \Ker(\HH_1(\KKK{g};\Q) \rightarrow \HH_1(\Sigma_g;\Q)) \quad \text{and} \quad \CCC{g,1} = \Ker(\HH_1(\KKK{g,1};\Q) \rightarrow \HH_1(\Sigma_{g,1};\Q)).$$
By Lemma \ref{lemma:decomposecover}, we have mapping class group invariant decompositions
$$\HH_1(\KKK{g};\Q) \cong \HH_1(\Sigma_{g};\Q) \oplus \CCC{g} \quad \text{and} \quad \HH_1(\KKK{g,1};\Q) \cong \HH_1(\Sigma_{g,1};\Q) \oplus \CCC{g,1}.$$
The surjective map $\pi_1(\Sigma_{g,1}) \rightarrow \pi_1(\Sigma_{g})$ induced by gluing a disc to the boundary
component of $\Sigma_{g,1}$ restricts to a surjection
$\KKK{g,1} \rightarrow \KKK{g}$.  Let $\III{g}$ be its kernel, so
we have a short exact sequence
\begin{equation}
\label{eqn:kkkiii}
0 \longrightarrow \III{g} \longrightarrow \KKK{g,1} \longrightarrow \KKK{g} \longrightarrow 0.
\end{equation}
We now prove the following.

\begin{lemma}
\label{lemma:identifyi}
For $g \geq 1$ and $L \geq 2$, the vector space $\III{g}$ is isomorphic as a $\Mod_{g,1}$-module
to the augmentation ideal of $\Q[H_L]$, where $H_L = \HH_1(\Sigma_{g,1};\Z/L)$.
\end{lemma}
\begin{proof}
Let $\ast \in \partial \Sigma_{g,1}$ be the basepoint and $\tilde{\ast} \in \CVR{g,1}$ be a lift
of $\ast$, so $K_{g,1} = \pi_1(\CVR{g,1},\tilde{\ast})$.  The subgroup
$\III{g} < \HH_1(\CVR{g,1};\Q)$ is exactly the subgroup generated by the homology classes of the
boundary components of $\CVR{g,1}$.
The group of deck transformations $H_L$ acts on these boundary components.  Let
$\gamma \in \pi_1(\Sigma_{g,1},\ast)$ be the simple closed curve that goes once around the boundary
component with the surface to its right.  Since $\gamma$
lifts to a simple closed curve $\tilde{\gamma} \in \pi_1(\CVR{g,1},\tilde{\ast})$,
the action of $H_L$ on the boundary components of $\tilde{\Sigma}$ is free.  For $v \in H_L$, let
$\KRHO{v} \in \III{g}$ denote the homology class of the $v$-translate of $\tilde{\gamma}$.
The only relation between the homology classes of the boundary components
of a surface with boundary is that their sum is $0$.  We conclude that $\III{g}$ is isomorphic to the
quotient of the $\Q$-vector space with basis the formal symbols
$\{\text{$\KRHO{v}$ $|$ $v \in H_L$}\}$ by the $1$-dimensional subspace generated by
$\sum_{v \in H_L} \KRHO{v}$.  This is exactly the augmentation ideal of $\Q[H_L]$, and we are done.
\end{proof}

Throughout the rest of this paper, we will denote by $\KRHO{v}$ the element of $\III{g}$ corresponding
to $v \in H_L$ as in the proof above.

\begin{remark}
It is clear that $\III{g} < \CCC{g,1}$, so the exact sequence \eqref{eqn:kkkiii} restricts to a short
exact sequence
$$0 \longrightarrow \III{g} \longrightarrow \CCC{g,1} \longrightarrow \CCC{g} \longrightarrow 0.$$
\end{remark}

\subsection{Actions on homology groups of abelian covers}

It is clear that $\Mod_{g,1}$ acts on $\KKK{g,1}$ and $\Mod_g^1$ acts on $\KKK{g}$.  Letting
$\beta$ be the boundary component of $\Sigma_{g,1}$, these
two mapping class groups are related by a short exact sequence
$$1 \longrightarrow \Z \longrightarrow \Mod_{g,1} \longrightarrow \Mod_g^1 \longrightarrow 1,$$
where $\Z$ is generated by $T_{\beta}$.  Since the loop around $\beta$ lies in $\KKK{g,1}$, the
action of $T_{\beta}$ on $\HH_1(\KKK{g,1};\Q)$ is trivial.  This implies that the action
of $\Mod_{g,1}$ on $\HH_1(\KKK{g,1};\Q)$ factors through an action of $\Mod_{g}^1$.

Let $\pi_1(\Sigma_g) < \Mod_g^1$ be the point-pushing subgroup.  The action of $\Mod_g^1$
on $\HH_1(\KKK{g};\Q)$ restricts to the action of $\pi_1(\Sigma_g)$ on $\HH_1(\KKK{g};\Q)$
induced by conjugation.  Restricting this action further to $\KKK{g} < \pi_1(\Sigma_g)$ thus
yields the trivial action.  However, if we instead restrict the action of $\Mod_g^1$ on
$\HH_1(\KKK{g,1};\Q)$ to $\KKK{g}$, we do {\em not} get a trivial action.

For $\gamma \in \KKK{g}$, denote by $\CLL{\gamma}$ the associated element of $\HH_1(\KKK{g};\Q)$.
Since $\KKK{g} < \Mod_g^1$ acts trivially on the kernel and cokernel of the short exact sequence
$$0 \longrightarrow \III{g} \longrightarrow \HH_1(\KKK{g,1};\Q) \stackrel{\rho}{\longrightarrow} \HH_1(\KKK{g};\Q) \longrightarrow 0,$$
the action of $\KKK{g}$ on $\HH_1(\KKK{g,1};\Q)$ is of the form
$$\gamma(x) = x + \REID{\CLL{\gamma}}{\rho(x)} \quad \quad \text{($\gamma \in \KKK{g}$ and $x \in \HH_1(\KKK{g,1};\Q)$)}$$
for some $\III{g}$-valued bilinear form $\REID{\cdot}{\cdot}$ on $\HH_1(\KKK{g};\Q)$.
This bilinear form has the following nice description.
Let $\langle \cdot, \cdot \rangle_{K}$ be the algebraic intersection pairing on
$\HH_1(\KKK{g};\Q) = \HH_1(\CVR{g};\Q)$.  The group $H_L = \HH_1(\Sigma_g;\Z/L)$ acts on $\HH_1(\CVR{g};\Q)$ via deck transformations.
The bilinear pairing in the following lemma first appeared in work of 
Reidemeister \cite{Reidemeister1, Reidemeister2} and has since been studied by
many people (see, e.g.,\ \cite{ChurchPixton, HempelIntersection, LooijengaPrym}).  We
will call it the {\em Reidemeister pairing}.

\begin{lemma}
\label{lemma:higherpairing}
For $\gamma \in \KKK{g}$ and $x \in \KKK{g}$, we have
$$\REID{\CLL{\gamma}}{x} = \sum_{v \in H_L} \langle v \cdot \CLL{\gamma}, x \rangle_{K} \KRHO{v}.$$
\end{lemma}
\begin{proof}
An immediate consequence of Lemma \ref{lemma:homologyaction}.
\end{proof}

\subsection{A key technical lemma}
\label{section:keylemma}

One of the linchpins of our proofs of our main theorems is the following lemma
about the action of $\KKK{g}$ on $\CCC{g,1} < \KKK{g,1}$.  Its proof
is lengthy and is given in \S \ref{section:mainlemma}.

\begin{lemma}
\label{lemma:mainlemma}
For $g \geq 4$ and $L \geq 2$, the map $\HH_1(\KKK{g};\CCC{g,1}) \rightarrow \HH_1(\Mod_{g}^1(L);\CCC{g})$
is the zero map.
\end{lemma}

\begin{remark}
The map $\HH_1(\KKK{g};\CCC{g,1}) \rightarrow \HH_1(\Mod_{g}^1(L);\CCC{g})$ factors through $\HH_1(\KKK{g};\CCC{g})$, and most
of our hard work is devoted to characterizing the image of $\HH_1(\KKK{g};\CCC{g,1})$ in $\HH_1(\KKK{g};\CCC{g})$.  It would
be much easier if we could instead prove that the map $\HH_1(\KKK{g};\CCC{g}) \rightarrow \HH_1(\Mod_g^1(L);\CCC{g})$ was the
zero map, but alas a careful examination of our proof of Theorem \ref{theorem:mainclosed} shows that this is not true.
\end{remark}

\noindent
In the course of proving Lemma \ref{lemma:mainlemma}, we will also prove the following.

\begin{lemma}
\label{lemma:killcg}
For $g \geq 3$ and $L \geq 2$, we have $(\CCC{g})_{\pi_1(\Sigma_g)} = (\CCC{g,1})_{\pi_1(\Sigma_g)} = 0$.
\end{lemma}

One useful consequence of Lemma \ref{lemma:killcg} is the following.

\begin{lemma}
\label{lemma:boundarytopuncture}
For $g \geq 3$ and $L \geq 2$, the natural map $\HH_1(\Mod_{g,1}(L);\CCC{g}) \rightarrow \HH_1(\Mod_{g}^1(L);\CCC{g})$
is an isomorphism.
\end{lemma}
\begin{proof}
Let $\beta$ be the boundary component of $\Sigma_{g,1}$.  We have a short exact sequence
$$1 \longrightarrow \Z \longrightarrow \Mod_{g,1}(L) \longrightarrow \Mod_g^1(L) \longrightarrow 1,$$
where $\Z = \Span{T_{\beta}}$.  The last 3 terms of the
associated 5-term exact sequence with coefficients $\CCC{g}$ are
$$(\CCC{g})_{\Mod_g^1(L)} \longrightarrow \HH_1(\Mod_{g,1}(L);\CCC{g}) \longrightarrow \HH_1(\Mod_g^1(L);\CCC{g}) \longrightarrow 0.$$
By Lemma \ref{lemma:killcg}, we have $(\CCC{g})_{\Mod_g^1(L)} = 0$, and the lemma follows.
\end{proof}

\section{Proofs of the main theorems}
\label{section:mainproofs}

We now turn to the proofs of our main theorems.  In this section, we will assume the truth 
of Lemmas \ref{lemma:mainlemma} and \ref{lemma:killcg}, which
are proven in subsequent sections.

\subsection{Surfaces with boundary}

We begin with Theorem \ref{theorem:mainboundary}, which
asserts that if $g \geq 4$ and $L \geq 2$, then
$$\HH_1(\Mod_{g,1};\HH_1(\KKK{g,1};\Q)) \cong \HH_1(\Mod_{g,1}(L);\HH_1(\KKK{g,1};\Q)) \cong \Q.$$
As was noted
in \S \ref{section:cast}, we can use Lemma \ref{lemma:decomposecover} to obtain a $\Mod_{g,1}$-invariant
decomposition
$$\HH_1(\KKK{g,1};\Q) \cong \HH_1(\Sigma_{g,1};\Q) \oplus \CCC{g,1}.$$
This implies that
$$\HH_1(\Mod_{g,1};\HH_1(\KKK{g,1};\Q)) \cong \HH_1(\Mod_{g,1};\HH_1(\Sigma_{g,1};\Q)) \oplus
\HH_1(\Mod_{g,1};\CCC{g,1}),$$
and similarly for $\Mod_{g,1}(L)$.  Theorem \ref{theorem:h1modlh} says that
$$\HH_1(\Mod_{g,1};\HH_1(\Sigma_{g,1};\Q)) \cong \HH_1(\Mod_{g,1}(L);\HH_1(\Sigma_{g,1};\Q)) \cong \Q.$$
To prove Theorem \ref{theorem:mainboundary}, therefore, it is enough to prove the following theorem.

\begin{theorem}
\label{theorem:mainboundary2}
For $g \geq 4$ and $L \geq 2$, we have
$$\HH_1(\Mod_{g,1};\CCC{g,1}) \cong \HH_1(\Mod_{g,1}(L);\CCC{g,1}) = 0.$$
\end{theorem}
\begin{proof}
Since $\Mod_{g,1}(L)$ is a finite-index subgroup of $\Mod_{g,1}$, Lemma \ref{lemma:transfer} implies that it is
enough to prove that $\HH_1(\Mod_{g,1}(L);\CCC{g,1}) = 0$.  Associated to the Birman exact sequence
$$1 \longrightarrow \pi_1(U\Sigma_g) \longrightarrow \Mod_{g,1}(L) \longrightarrow \Mod_{g}(L) \longrightarrow 1$$
is a 5-term exact sequence in homology with coefficients in $\CCC{g,1}$.  The last 3 terms of this are
$$(\HH_1(\pi_1(U\Sigma_g);\CCC{g,1}))_{\Mod_{g,1}(L)} \stackrel{f}{\longrightarrow} \HH_1(\Mod_{g,1}(L);\CCC{g,1}) \longrightarrow \HH_1(\Mod_g(L);(\CCC{g,1})_{\pi_1(U\Sigma_g)}) \longrightarrow 0.$$
Lemma \ref{lemma:killcg} says that
$$(\CCC{g,1})_{\pi_1(U\Sigma_g)} = (\CCC{g,1})_{\pi_1(\Sigma_g)} = 0.$$
To prove the theorem, therefore, it is enough to show that $f = 0$.  This is equivalent
to showing that the map
$$f' : \HH_1(\pi_1(U\Sigma_g);\CCC{g,1}) \rightarrow \HH_1(\Mod_{g,1}(L);\CCC{g,1})$$
is the zero map.

Our goal is to deduce the fact that $f'=0$ from Lemma \ref{lemma:mainlemma}.
To do this, we perform a series of reductions.  From the short exact sequence
$$0 \longrightarrow \III{g} \longrightarrow \CCC{g,1} \longrightarrow \CCC{g} \longrightarrow 0$$
of $\Mod_{g,1}(L)$-modules we obtain a long exact sequence in homology.  This long exact sequence
contains the segment
$$\HH_1(\Mod_{g,1}(L);\III{g}) \longrightarrow \HH_1(\Mod_{g,1}(L);\CCC{g,1}) \longrightarrow \HH_1(\Mod_{g,1}(L);\CCC{g}).$$
Since $\Mod_{g,1}(L)$ acts trivially on $\III{g}$, Theorem \ref{theorem:h1modl} implies that
$\HH_1(\Mod_{g,1}(L);\III{g})=0$.  It follows that the map
$\HH_1(\Mod_{g,1}(L);\CCC{g,1}) \rightarrow \HH_1(\Mod_{g,1}(L);\CCC{g})$
is injective.  Lemma \ref{lemma:boundarytopuncture} says that
$\HH_1(\Mod_{g,1}(L);\CCC{g}) \cong \HH_1(\Mod_{g}^1(L);\CCC{g})$, so we deduce that
to prove that $f'=0$, it is enough to prove that the map
$$\HH_1(\pi_1(U\Sigma_g);\CCC{g,1}) \longrightarrow \HH_1(\Mod_{g}^1(L);\CCC{g})$$
is the zero map.  This map factors through the map
\begin{equation}
\label{eqn:mapzero}
\HH_1(\pi_1(\Sigma_g);\CCC{g,1}) \longrightarrow \HH_1(\Mod_{g}^1(L);\CCC{g}).
\end{equation}
Since $\KKK{g}$ is a finite-index subgroup of $\pi_1(\Sigma_g)$, Lemma \ref{lemma:transfer}
says that the map
$\HH_1(\KKK{g};\CCC{g,1}) \longrightarrow \HH_1(\pi_1(\Sigma_g);\CCC{g,1})$
is surjective.  Thus to show that the map in \eqref{eqn:mapzero} vanishes, it is enough to
show that the map
$$\HH_1(\KKK{g};\CCC{g,1}) \rightarrow \HH_1(\Mod_{g}^1(L);\CCC{g})$$
vanishes, which is exactly the content of Lemma \ref{lemma:mainlemma}.
\end{proof}

\subsection{Closed surfaces}

We now turn to Theorems \ref{theorem:mainmcg} and \ref{theorem:mainclosed}, which assert
that if $g \geq 4$ and $L \geq 2$, then
$$\HH_1(\Mod_{g}^1;\HH_1(\KKK{g};\Q)) \cong \Q^{\tau(L)} \quad \text{and} \quad \HH_1(\Mod_{g}^1(L);\HH_1(\KKK{g};\Q)) \cong \Q[H_L].$$
Here $\tau(L)$ is the number of positive divisors of $L$ and $H_L \cong \HH_1(\Sigma_g;\Z/L)$.  Also, the
second isomorphism should be equivariant with respect to $\Mod_{g}^1$ actions on $\HH_1(\Mod_{g}^1(L);\HH_1(\KKK{g};\Q))$
and $\Q[H_L]$.  We begin by deriving Theorem \ref{theorem:mainmcg} from Theorem \ref{theorem:mainclosed}.

\begin{proof}[{Proof of Theorem \ref{theorem:mainmcg}, assuming Theorem \ref{theorem:mainclosed}}]
Lemma \ref{lemma:supertransfer} implies that
$$\HH_1(\Mod_{g}^1;\HH_1(\KKK{g};\Q)) \cong (\HH_1(\Mod_{g}^1(L);\HH_1(\KKK{g};\Q)))_{\Mod_{g}^1}.$$
Applying Theorem \ref{theorem:mainclosed}, we must show that $(\Q[H_L])_{\Mod_{g}^1} \cong \Q^{\tau(L)}$.  The
vector space $\Q[H_L]$ has a basis that is permuted by the action of $\Mod_{g}^1$, namely the elements of $H_L$.
It is enough, therefore, to show that there are $\tau(L)$ orbits of the action of $\Mod_{g}^1$ on $H_L$.  This action
factors through the surjection $\Mod_{g}^1 \rightarrow \Sp_{2g}(\Z/L)$.  Let $v \in H_L$ be a fixed primitive vector, and
set
$$X = \{\text{$cv$ $|$ $c$ is a positive divisor of $L$}\} \subset H_L.$$
The set $X$ has cardinality $\tau(L)$, and clearly no two elements of $X$ are in the same $\Sp_{2g}(\Z/L)$-orbit.
Also, if $w \in H_L$,
then there is a primitive vector $w'$ and a positive divisor $c$ of $L$ such that $w = c w'$.  Since $\Sp_{2g}(\Z/L)$ acts
transitively on the set of primitive vectors, there is some $\phi \in \Sp_{2g}(\Z/L)$ such that $v = \phi(w')$.  Thus
$w$ is in the same $\Sp_{2g}(\Z/L)$-orbit as $c v \in X$.  We conclude that $X$ contains a unique representative from every
$\Sp_{2g}(\Z/L)$-orbit, and we are done.
\end{proof}

Finally, we discuss Theorem \ref{theorem:mainclosed}.
Lemma \ref{lemma:decomposecover} implies that there is a $\Mod_g^1$-invariant decomposition
$$\HH_1(\KKK{g};\Q) \cong \HH_1(\Sigma_{g};\Q) \oplus \CCC{g},$$
so we have a $\Mod_{g}^1$-invariant decomposition
$$\HH_1(\Mod_{g}^1(L);\HH_1(\KKK{g};\Q)) \cong \HH_1(\Mod_{g}^1(L);\HH_1(\Sigma_{g};\Q)) \oplus
\HH_1(\Mod_{g}^1(L);\CCC{g}).$$
Also, Theorem \ref{theorem:h1modlh} says that
$$\HH_1(\Mod_{g}^1(L);\HH_1(\Sigma_{g};\Q)) \cong \Q,$$
so we obtain a $\Mod_{g}^1$-invariant decomposition
$$\HH_1(\Mod_{g}^1(L);\HH_1(\KKK{g};\Q)) \cong \Q \oplus \HH_1(\Mod_{g}^1(L);\CCC{g}).$$
By Lemma \ref{lemma:groupalgebra}, we have $\Q[H_L] \cong \Q \oplus I$, where $I$ is the augmentation ideal
of $\Q[H_L]$.  Lemma \ref{lemma:identifyi} says that $\III{g}$ is isomorphic to the augmentation ideal
of $\Q[H_L]$, so we conclude that it is enough to prove the following theorem.

\begin{theorem}
For $g \geq 4$ and $L \geq 2$, we have a $\Mod_{g}^1$-equivariant isomorphism
$\HH_1(\Mod_{g}^1(L);\CCC{g}) \cong \III{g}$.
\end{theorem}
\begin{proof}
Lemma \ref{lemma:boundarytopuncture} says that there is an isomorphism
\begin{equation}
\label{eqn:boundarytopuncture}
\HH_1(\Mod_{g,1}(L);\CCC{g}) \cong \HH_1(\Mod_{g}^1(L);\CCC{g}).
\end{equation}
The action of $\Mod_{g,1}$ on $\HH_1(\Mod_{g,1}(L);\CCC{g})$ factors through $\Mod_{g}^1$, and it is easy to see
that the isomorphism in \eqref{eqn:boundarytopuncture} is $\Mod_{g}^1$-equivariant.
We deduce that it is enough to construct a $\Mod_{g,1}$-equivariant isomorphism $\HH_1(\Mod_{g,1}(L);\CCC{g}) \cong \III{g}$.
The long exact sequence in $\Mod_{g,1}(L)$-homology associated to the short exact sequence
$$0 \longrightarrow \III{g} \longrightarrow \CCC{g,1} \longrightarrow \CCC{g} \longrightarrow 0$$
of $\Mod_{g,1}(L)$-modules contains the segment
$$\HH_1(\Mod_{g,1}(L);\CCC{g,1}) \longrightarrow \HH_1(\Mod_{g,1}(L);\CCC{g}) \longrightarrow \III{g} \longrightarrow (\CCC{g,1})_{\Mod_{g,1}(L)}.$$
Here we are using the fact that $\Mod_{g,1}(L)$ acts trivially on $\III{g}$, so
$$\HH_0(\Mod_{g,1}(L);\III{g}) = (\III{g})_{\Mod_{g,1}(L)} = \III{g}.$$
Theorem \ref{theorem:mainboundary2} says that $\HH_1(\Mod_{g,1}(L);\CCC{g,1}) = 0$, and Lemma \ref{lemma:killcg} implies
that $(\CCC{g,1})_{\Mod_{g,1}(L)} = 0$.  We obtain an isomorphism $\HH_1(\Mod_{g,1}(L);\CCC{g}) \cong \III{g}$, which
is easily verified to be $\Mod_{g,1}$-equivariant.  The theorem follows.
\end{proof}

\section{Skeleton of the proof of Lemma \ref{lemma:mainlemma}}
\label{section:mainlemma}

This section sets the stage for the remainder of the paper by reducing
the proof of Lemma \ref{lemma:mainlemma} to two further lemmas which
are proven in the remaining sections.  It also
proves Lemma \ref{lemma:killcg}.  This is all done
in \S \ref{section:mainlemmaproof}.  Before that, \S \ref{section:notation}
sets up some notation which will be used in the remainder of the paper.

\subsection{Notation for $\KKK{g}$}
\label{section:notation}

In this section, we will introduce notation for elements of $\HH_1(\KKK{g};\Q)$.  Let $\rho : \KKK{g} \rightarrow \HH_1(\KKK{g};\Q)$
be the abelianization map.  Recall that the notation $g_1^{g_2}$ stands for $g_2 g_1 g_2^{-1}$.  Observe that
if $a \in \KKK{g}$ and $x,y \in \pi_1(\Sigma_g)$ are such that $x y^{-1} \in \KKK{g}$, then
$$\rho(a^x) = \rho(a^{(x y^{-1})y}) = \rho((xy^{-1})(a^y)(xy^{-1})^{-1}) = \rho(xy^{-1})+\rho(a^y)-\rho(xy^{-1}) = \rho(a^y).$$
This implies that for $a \in \KKK{g}$ and $x,y \in \pi_1(\Sigma_g)$ and $v \in H_L$, we may unambiguously define
$$\CL{a}{v} = \rho(a^{\tilde{v}}) \quad \quad \text{and} \quad \quad \CL{x,y}{v} = \rho([x,y]^{\tilde{v}}),$$ 
where $\tilde{v} \in \pi_1(\Sigma_g)$ is any lift of $v$ under the
map $\pi_1(\Sigma_g) \rightarrow H_L$.  

As another bit of notation, for $x \in \pi_1(\Sigma_g)$ denote by $\CAB{x}$ the associated element of 
$H_L = \HH_1(\Sigma_g;\Z/L)$.

With this notation, we have the following identities. 

\begin{lemma}
\label{lemma:commutatoridentities}
\mbox{}
\begin{enumerate}
\item For $a \in \KKK{g}$ and $x \in \pi_1(\Sigma_g)$, we have $\CLL{a,x} = \CLL{a} - \CL{a}{\CAB{x}}$.
\item For $x,y \in \pi_1(\Sigma_g)$, we have $\CLL{x,y} = -\CLL{y,x}$.
\item For $x,y,z \in \pi_1(\Sigma_g)$, we have $\CLL{xy,z} = \CLL{x,z} + \CL{y,z}{\CAB{x}}$.  
\item For $x,y \in \pi_1(\Sigma_g)$, we have $\CLL{x^{-1},y} = - \CL{x,y}{-\CAB{x}}$.
\end{enumerate}
\end{lemma}
\begin{proof}
Items 1 and 2 are obvious, item 3 follows from the commutator identity $[xy,z] = [x,z][y,z]^x$, and
item 4 follows from item 3.
\end{proof}

\subsection{Skeleton of the proof of Lemma \ref{lemma:mainlemma}}
\label{section:mainlemmaproof}

Recall that we want to prove that the map
$$\HH_1(\KKK{g};\CCC{g,1}) \rightarrow \HH_1(\Mod_{g}^1(L);\CCC{g})$$
is the zero map.  This map factors through $\HH_1(\KKK{g};\CCC{g})$.  Since
$\KKK{g}$ acts trivially on $\CCC{g}$, this latter group is isomorphic
to $\HH_1(\KKK{g};\Z) \otimes \CCC{g}$.

Our first order of business is
to characterize the image of $\HH_1(\KKK{g};\CCC{g,1})$ in $\HH_1(\KKK{g};\CCC{g})$.
The long exact sequence in $\KKK{g}$-homology associated to the short exact sequence
$$0 \longrightarrow \III{g} \longrightarrow \CCC{g,1} \longrightarrow \CCC{g} \longrightarrow 0$$
of $\KKK{g}$-modules contains the segment
\begin{equation}
\label{eqn:kkklong}
\HH_1(\KKK{g};\CCC{g,1}) \longrightarrow \HH_1(\KKK{g};\Z) \otimes \CCC{g} \stackrel{\partial}{\longrightarrow} \III{g}.
\end{equation}
Here we have used the fact that $\HH_0(\KKK{g};\III{g}) = (\III{g})_{\KKK{g}} = \III{g}$.
The image of $\HH_1(\KKK{g};\CCC{g,1})$ in $\HH_1(\KKK{g};\CCC{g})$ is the kernel of $\partial$.
It turns out that $\partial$ is closely related to the Reidemeister pairing $\REID{\cdot}{\cdot}$ from
Lemma \ref{lemma:higherpairing}.

\begin{lemma}
\label{lemma:boundary}
The map $\partial : \HH_1(\KKK{g};\Z) \otimes \CCC{g} \rightarrow \III{g}$ is the
restriction of the linear map $\HH_1(\KKK{g};\Q) \otimes \HH_1(\KKK{g};\Q) \rightarrow \III{g}$
induced by the bilinear Reidemeister pairing $\REID{\cdot}{\cdot}$.
\end{lemma}
\begin{proof}
Consider $f \otimes y \in \HH_1(\KKK{g};\Z) \otimes \CCC{g}$.  Tracing through
the construction of the long exact sequence (see, e.g.,\ \cite[\S III.7]{BrownCohomology}), we can calculate
$\partial(f \otimes y)$ as follows.  Lift $y$ to $\tilde{y} \in \CCC{g,1}$.  Then
$$\partial(f \otimes y) = f(\tilde{y}) - \tilde{y} \in \III{g}.$$
The lemma then follows from Lemma \ref{lemma:higherpairing}.
\end{proof}

We thus want to determine which pairs of elements evaluate to zero under $\REID{\cdot}{\cdot}$.
This will require the following definitions.

\begin{definition}
Consider sets of curves $s = \{x_1,\ldots,x_n\} \subset \pi_1(\Sigma_g)$ and
$s' = \{y_1,\ldots,y_m\} \subset \pi_1(\Sigma_g)$.
\begin{itemize}
\item The sets $s$ and $s'$ are {\em essentially separate} if there exist connected subsurfaces
$X$ and $X'$ of $\Sigma_g$ with the following properties.
\begin{itemize}
\item Both $X$ and $X'$ contain the basepoint.
\item $\Sigma_g = X \cup X'$ and $X \cap X' = \partial X = \partial X'$.
\item $s \subset \Image(\pi_1(X) \rightarrow \pi_1(\Sigma_g))$ and
$s' \subset \Image(\pi_1(X') \rightarrow \pi_1(\Sigma_g))$.
\end{itemize}
\item The sets $s$ and $s'$ are {\em strongly essentially separate} if we can choose $X$ and $X'$
as above such that both $X$ and $X'$ have exactly one boundary component (which necessarily
contains the basepoint).
\end{itemize}
\end{definition}

\noindent
See Figure \ref{figure:essentiallyseparate} for examples of curves that are essentially
separate and strongly essentially separate.

This brings us to the following key definition.  Recall that $H_L = \HH_1(\Sigma_g;\Z/L)$.

\Figure{figure:essentiallyseparate}{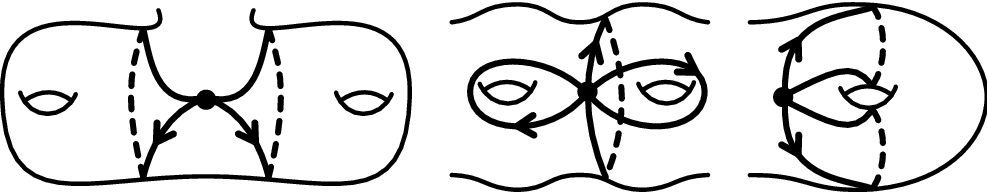}{a. $x$ and $y$ are strongly essentially separate
\CaptionSpace b. $\{x,y\}$ is strongly essentially separate from $\{x,z\}$.
\CaptionSpace c. $x$ and $y$ are essentially separate but not strongly essentially separate}

\begin{definition}
Define $\SSS \subset \HH_1(\KKK{g};\Q) \times \CCC{g}$ to equal $\SSSS{1} \cup \SSSS{2}$,
where the $\SSSS{i}$ are as follows.  To simplify our notation, we will denote $\pi_1(\Sigma_g)$ by $\pi$.
\begin{align*}
\SSSS{1} &= \Set[,]{$(\CL{x}{v}, \CL{y}{v'})$}{$v,v' \in H_L$, $x \in \KKK{g}$, $y \in [\pi,\pi]$, and $x$ and $y$ are essentially separate}\\
\SSSS{2} &= \Set[.]{$(\CL{x}{v}, \CL{y,z^L}{v'})$}{$v,v' \in H_L$, $x \in \KKK{g}$, $y,z \in \pi$,
$z$ can be realized by a simple\\ closed nonseparating curve, and
$\{z\}$ and $\{x,y\}$ are\\ strongly essentially separate}
\end{align*}
\end{definition}

\noindent
We now prove the following.

\begin{lemma}
\label{lemma:reidvanish}
For $(x,y) \in \SSS$, we have $\REID{x}{y}=0$.
\end{lemma}
\begin{proof}
We must deal with both $\SSSS{1}$ and $\SSSS{2}$.

\BeginSteps
\begin{step}
$\REID{\CL{x}{v}}{\CL{y}{v'}} = 0$ for $(\CL{x}{v}, \CL{y}{v'}) \in \SSSS{1}$.
\end{step}

Since $x$ and $y$ are essentially separate, they can be freely homotoped to disjoint curves.  This
implies that any two lifts of $x$ and $y$ to the cover of $\Sigma_g$ corresponding to $\KKK{g}$
can be homotoped so as to be disjoint. 
Examining the formula for $\REID{\cdot}{\cdot}$ in Lemma \ref{lemma:higherpairing}, this
immediately implies that $\REID{\CL{x}{v}}{\CL{y}{v'}} = 0$, as desired.

\begin{step}
$\REID{\CL{x}{v}}{\CL{y,z^L}{v'}}=0$ for $(\CL{x}{v}, \CL{y,z^L}{v'}) \in \SSSS{2}$.
\end{step}

Recall that if $w \in \pi_1(\Sigma_g)$, then $\CAB{w}$ denotes
the element of $H_L = \HH_1(\Sigma_g;\Z/L)$ associated to $w$.  By Lemma \ref{lemma:commutatoridentities}, we have
$$\CL{y,z^L}{v'} = \CL{z^L}{v'+\CAB{y}} - \CL{z^L}{v'}.$$
Since $x$ and $z^L$ are essentially separate, an argument similar to the argument in Step 1
shows that
$$\REID{\CL{x}{v}}{\CL{y,z^L}{v'}} = \REID{\CL{x}{v}}{\CL{z^L}{v'+\CAB{y}}} - \REID{\CL{x}{v}}{\CL{z^L}{v'}} = 0 - 0 = 0,$$
as desired.
\end{proof}

\begin{remark}
In Step 2 of the above, we only used the fact that $z$ is essentially separate from $x$.  The
remainder of the assumptions on elements of $\SSSS{2}$ will be later used in the proof of
Lemma \ref{lemma:killsker}.
\end{remark}

It follows that the set $\{\text{$x \otimes y$ $|$ $(x,y) \in \SSS$}\}$ is contained in
$\ker(\partial)$.  This is not everything, however.  Since $\Mod_{g}^1(L)$ acts trivially
on $\III{g}$, the kernel of $\partial$ also contains $x \otimes y - f(x) \otimes f(y)$ for
$x \in \HH_1(\KKK{g};\Z)$ and $y \in \CCC{g}$ and $f \in \Mod_g^1(L)$.
Define
$\SKER < \HH_1(\KKK{g};\Z) \otimes \CCC{g}$ be the span of the set
$$\{\text{$x \otimes y$ $|$ $(x,y) \in \SSS$}\} \cup \{\text{$x \otimes y - f(x) \otimes f(y)$ $|$ $x \in \HH_1(\KKK{g};\Z)$, $y \in \CCC{g}$, $f \in \Mod_g^1(L)$}\}$$
and let $\QQQ{g} = \HH_1(\KKK{g};\Z) \otimes \CCC{g} / \SKER$.  Since $\SKER \subset \Ker(\partial)$, the map
$\partial : \HH_1(\KKK{g};\Z) \otimes \CCC{g} \rightarrow \III{g}$ induces a map
$\psi : \QQQ{g} \rightarrow \III{g}$.  Following preliminary results in
\S \ref{section:generatorsandrelations}, we will prove the following lemma in \S \ref{section:linearalgebra}.

\begin{lemma}
\label{lemma:psiisomorphism}
For $g \geq 3$, the map $\psi$ is an isomorphism.
\end{lemma}

It follows that
$$\SKER = \Ker(\partial) = \Image(\HH_1(\KKK{g};\CCC{g,1}) \rightarrow \HH_1(\KKK{g};\CCC{g})).$$
To prove Lemma \ref{lemma:mainlemma}, which asserts that the map
$\HH_1(\KKK{g};\CCC{g,1}) \rightarrow \HH_1(\Mod_g^1(L);\CCC{g})$ is the zero map, it therefore
suffices to prove the following lemma, whose proof is in \S \ref{section:kills}.

\begin{lemma}
\label{lemma:killsker}
For $g \geq 4$, the image of $\SKER$ in $\HH_1(\Mod_g^1(L);\CCC{g})$ is zero.
\end{lemma}

This completes the outline of the proof of Lemma \ref{lemma:mainlemma} (and our outline of the
remainder of the paper).  However, we also owe the reader a proof of Lemma \ref{lemma:killcg},
which asserts that
$$(\CCC{g})_{\pi_1(\Sigma_g)} = (\CCC{g,1})_{\pi_1(\Sigma_g)} = 0$$
for $g \geq 3$.

\begin{proof}[{Proof of Lemma \ref{lemma:killcg}}]
We can extend the long exact sequence \eqref{eqn:kkklong} to the right to get an exact sequence
$$\HH_1(\KKK{g};\Z) \otimes \CCC{g} \stackrel{\partial}{\longrightarrow} \III{g} \longrightarrow (\CCC{g,1})_{\KKK{g}} \longrightarrow \CCC{g} \longrightarrow 0.$$
Lemma \ref{lemma:psiisomorphism} implies that $\partial$ is surjective, so we deduce that
$(\CCC{g,1})_{\KKK{g}} \cong \CCC{g}$.  It is thus enough to prove that $(\CCC{g})_{\pi_1(\Sigma_g)} = 0$.
Lemma \ref{lemma:decomposecover} implies that there is a $\pi_1(\Sigma_g)$-invariant decomposition
$$\HH_1(\KKK{g};\Q) \cong \CCC{g} \oplus \HH_1(\Sigma_g;\Q).$$
Also, Lemma \ref{lemma:supertransfer} implies that
$(\HH_1(\KKK{g};\Q))_{\pi_1(\Sigma_g)} \cong \HH_1(\Sigma_g;\Q)$.  We conclude that
$(\CCC{g})_{\pi_1(\Sigma_g)} = 0$.
\end{proof}

\section{Generators and relations for $\QQQ{g}$}
\label{section:generatorsandrelations}

This section contains preliminaries for the proof of Lemma \ref{lemma:psiisomorphism}.
We begin in \S \ref{section:intersectionpattern} by introducing the notion of the intersection pattern
of curves, which will play an important role in both this section and in \S \ref{section:kills}.  Next,
in \S \ref{section:x}, we introduce certain important elements $X(v,w_1,w_2)$ of $\QQQ{g}$.  We
calculate the image of $X(v,w_1,w_2)$ under $\psi$ in \S \ref{section:xpsiimage}.  We show
that the $X(v,w_1,w_2)$ span $\QQQ{g}$ in \S \ref{section:qgenerators1}.  Finally, in
\S \ref{section:qrelations}, we determine some relations between these elements.

Lemma \ref{lemma:psiisomorphism} is proven in \S \ref{section:linearalgebra} below.  The proof is
essentially a lengthy calculation with generators and relations.

Throughout this section, let $\eta : \HH_1(\KKK{g};\Z) \otimes \CCC{g} \rightarrow \QQQ{g}$
be the projection.

\subsection{Intersection patterns}
\label{section:intersectionpattern}

We will have to perform some detailed calculations with elements of $\pi_1(\Sigma_g)$.  These calculations
will depend on certain pictures of curves on the surface, and in this section we will establish
some vocabulary for this.  We begin with the following definition.

\begin{definition}
Let $S$ be a surface, possibly with boundary.  Assume that $S$ has a fixed basepoint.
An embedding $i : S \rightarrow \Sigma_g$ is a {\em simple embedding} if $i$ takes
the basepoint of $S$ to the basepoint of $\Sigma_g$ and 
all components of $\Sigma_g \setminus i(S)$ have one boundary component.
\end{definition}

\begin{remark}
We allow $S = \Sigma_g$ and $i=\text{id}$.
\end{remark}

\begin{remark}
The key property of simple embeddings is as follows.  Let $i : S \rightarrow \Sigma_g$ be a simple
embedding.  Then if $\gamma$ is a simple closed separating
curve on $S$, then $i(\gamma)$ is a simple closed separating curve on $\Sigma_g$.
\end{remark}

Next, we make the following definition.

\begin{definition}
Let $S$ be a surface, possibly with boundary, and let  $\{x_1',\ldots,x_k'\} \subset \pi_1(S)$.
We will say that a set $\{x_1,\ldots,x_k\} \subset \pi_1(\Sigma_g)$
of curves {\em has the same unoriented intersection pattern} as $\{x_1',\ldots,x_k'\}$ 
if there is a simple embedding $f:S \rightarrow \Sigma_g$ such that $f_{\ast}(x_i') = x_i^{\pm 1}$
for all $1 \leq i \leq k$.  If $f$ can be chosen such that $f(x_i') = x_i$ for all $1 \leq i \leq k$,
then we will say that that the curves {\em have the same oriented intersection pattern}.
\end{definition}

\begin{remark}
In what follows, the surface $S$ and the curves $\{x_1',\ldots,x_k'\}$ will often
be given by pictures.  To avoid cluttering the pictures, we will often depict boundary
components via gaps in their edges.  For instance, there are boundary components at
the top and bottom of Figure \ref{figure:xwelldefined}.a below.
\end{remark}

\Figure{figure:intersectionpattern}{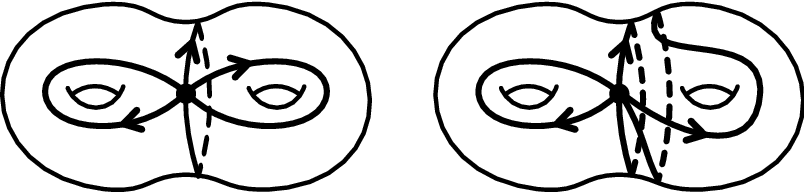}{The curves $\{x,y,z\}$ in a have the same
oriented intersection pattern as the curves $\{x,y',z\}$ in b.}

We will frequently assert without proof that a set of curves has a given (un)oriented
intersection pattern. In all these cases, the assertion will be a trivial consequence of
the ``change of coordinates'' principle from \cite[\S 1.3]{FarbMargalitPrimer}.  Rather
than give a formal description of this principle, we will illustrate it with a concrete example (there
are many more examples in \cite[\S 1.3]{FarbMargalitPrimer}).  Namely, we will prove
that the curves $\{x,y,z\}$ in Figure \ref{figure:intersectionpattern}.a have the same
oriented intersection pattern as the curves $\{x,y',z\}$ in Figure \ref{figure:intersectionpattern}.b (we
remark that $y' = y^{-1} x$).  

The proof is as follows.  The union of the curves in
Figure \ref{figure:intersectionpattern}.a (resp.\ Figure \ref{figure:intersectionpattern}.b)
forms an oriented graph $\Gamma_1$ (resp.\ $\Gamma_2$) embedded in $\Sigma_2$ with one vertex (the basepoint)
and three loops labeled with $\{x,y,z\}$ (resp.\ $\{x,y',z\}$).  
There is an isomorphism $f: \Gamma_1 \rightarrow \Gamma_2$ taking the edge labeled $x$ to the
edge labeled $x$, the edge labeled $y$ to the edge labeled $y'$, and the edge labeled $z$ to 
the edge labeled $z$.  The
embedding of $\Gamma_i$ in $\Sigma_2$ induces a cyclic order on the oriented edges entering
and leaving the single vertex, and the isomorphism $f$ respects these cyclic orderings.  This
implies that $f$ extends to a diffeomorphism $f' : N_1 \rightarrow N_2$, where $N_i$ is a regular
neighborhood of $\Gamma_i$.  An Euler characteristic computation shows that the components
of $\Sigma_2 \setminus N_i$ are diffeomorphic, and we thus obtain a basepoint-preserving diffeomorphism 
$f'' : \Sigma_2 \rightarrow \Sigma_2$ such that $(f'')_{\ast}(x) = x$ and $(f'')_{\ast}(y)=y'$
and $(f'')_{\ast}(z)=z$, as desired.

\subsection{The elements $X(v,w_1,w_2)$}
\label{section:x}

The purpose of this section is to introduce certain elements $X(v,w_1,w_2)$ in $\QQQ{g}$.  The
key will be the following lemma.  In it, recall that if $x \in \pi_1(\Sigma_g)$, then $\CAB{x}$ denotes
the element of $H_L = \HH_1(\Sigma_g;\Z/L)$ associated to $x$.

\Figure{figure:xwelldefined}{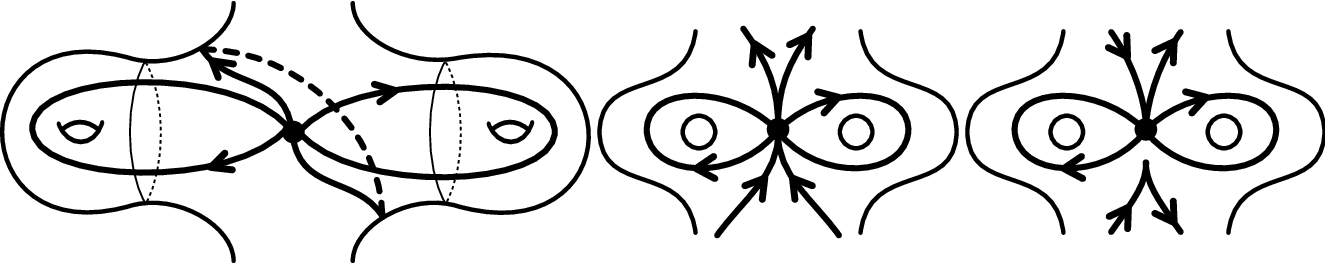}{a. The configuration of curves such that
$\phi(\CLL{x} \otimes \CL{y,z}{v}) = X(v,w_1,w_2)$.  Also, the one-holed tori $B_1$ and $B_2$ will be needed in the
proof of Lemma \ref{lemma:xpsiimage}.  The central four-holed sphere in the picture will be called $A$ in that proof.
\CaptionSpace b. The curves $f(x_1)$ and $x_2$ leave at the top and come back at the bottom.
\CaptionSpace c. The product $f(x_1) x_2^{-1}$ is essentially disjoint from $[y_2,z_2]$.  The orientations
of the ``top'' and ``bottom'' piece depend on the manner in which $f(x_1)$ and $x_2$ leave and
come back to the basepoint}

\begin{lemma}
\label{lemma:xwelldefined}
Fix $w \in H_L = \HH_1(\Sigma_g;\Z/L)$.  For $1 \leq i \leq 2$, let $x_i,y_i,z_i \in \pi_1(\Sigma_g)$ be such that $\{x_i,y_i,z_i\}$
has the same oriented intersection pattern as the curves $\{x,y,z\}$ in Figure \ref{figure:xwelldefined}.a.  Assume that
$\CAB{y}_1 = \CAB{y}_2$ and $\CAB{z}_1 = \CAB{z}_2$.  Then
$\eta(\CLL{x_1} \otimes \CL{y_1,z_1}{w}) = \eta(\CLL{x_2} \otimes \CL{y_2,z_2}{w})$.
\end{lemma}
\begin{proof}
It is easy to see that there exists some $f \in \Mod_g^1(L)$ such that $f(y_1) = y_2$ and $f(z_1) = z_2$
(the proof of this is a slight variation on the proof of \cite[Proposition 6.7]{PutmanSecondHomologyLevel}, which
proves the analogous result for unbased curves).  We then have
$$\eta(\CLL{x_1} \otimes \CL{y_1,z_1}{w}) = \eta(\CLL{f(x_1)} \otimes \CL{f(y_1),f(z_1)}{w}) = \eta(\CLL{f(x_1)} \otimes \CL{y_2,z_2}{w}).$$
The curves $\{f(x_1),y_2,z_2\}$ have the same oriented intersection pattern as the curves $\{x,y,z\}$ in
Figure \ref{figure:xwelldefined}.a.  Moreover (see Figures \ref{figure:xwelldefined}.b--c), the curves
$f(x_1) x_2^{-1}$ and $[y_2,z_2]$ are essentially separate, so we conclude that
$$\eta(\CLL{f(x_1)} \otimes \CL{y_2,z_2}{w}) = \eta(\CLL{x_2} \otimes \CL{y_2,z_2}{w}),$$
as desired.
\end{proof}

We will need the following definition.  Let $i(\cdot,\cdot)$ be the $\Z/L$-valued algebraic
intersection pairing on $H_L$.

\begin{definition}
A $k$-element set $\{w_1,\ldots,w_k\} \subset H_L$ will be said to be
{\em isotropic} if $i(w_i,w_j)=0$ for all $1 \leq i,j \leq k$ and {\em unimodular} if $\Span{w_1,\ldots,w_k}$
is direct summand of $H_L$ that is isomorphic to a $k$-dimensional free $\Z/L$-submodule.
\end{definition}

\noindent
It is clear that if the curves $\{x,y,z\}$ have the same oriented intersection pattern
as the curves in Figure \ref{figure:xwelldefined}.a, then $\{\CAB{y},\CAB{z}\} \subset H_L$ is
isotropic and unimodular.  The converse is true as well.  This will require the following lemma.

\begin{lemma}
\label{lemma:realizehomology}
For some $n,m \geq 0$, let $\{w_1,\ldots,w_n,w_1',\ldots,w_m'\} \subset H_L$ be a unimodular set.  Assume that
$$i(w_i,w_j) = i(w_{i'}',w_{j'}') = 0 \quad \text{and} \quad i(w_i,w_{i'}) = 
\begin{cases}
1 & \text{if $i=i'$}\\
0 & \text{if $i \neq i'$}
\end{cases}$$
for $1 \leq i,j \leq n$ and $1 \leq i',j' \leq m$.  There then exists a set $\{\alpha_1,\ldots,\alpha_n,\alpha_1',\ldots,\alpha_m'\}$ of unbased
oriented simple closed curves on $\Sigma_g$ with the
following properties.
\begin{itemize}
\item The $\Z/L$-homology class of $\alpha_i$ is $w_i$ for $1 \leq i \leq n$ and the $\Z/L$-homology
class of $\alpha_i'$ is $w_i'$ for $1 \leq i \leq m$.
\item $\alpha_i$ and $\alpha_{i}'$ intersect once for $1 \leq i \leq \Min(n,m)$.  Otherwise,
the curves $\{\alpha_1,\ldots,\alpha_n,\alpha_1',\ldots,\alpha_m'\}$ are pairwise disjoint.
\end{itemize}
\end{lemma}
\begin{proof}
Identical to the proof of \cite[Lemma A.3]{PutmanCutPaste}.
\end{proof}

If $\{w_1,w_2\} \subset H_L$ is isotropic and unimodular, then Lemma \ref{lemma:realizehomology}
says that we can find unbased, disjoint simple closed curves $Y$ and $Z$ such that the $\Z/L$-homology
classes of $Y$ and $Z$ are $w_1$ and $w_2$, respectively.  Connecting $Y$ and $Z$ to the basepoint
in an appropriate way, we find $y,z \in \pi_1(\Sigma_g)$ such that $\CAB{y}=w_1$ and $\CAB{z}=w_2$ and
$\{y,z\}$ has the same oriented intersection pattern as the curves in Figure \ref{figure:xwelldefined}.a.  It
is then clear that we can find some $x \in \pi_1(\Sigma_g)$ such that $\{x,y,z\}$ has the same
oriented intersection pattern as the curves in Figure \ref{figure:xwelldefined}.a.

We now introduce notation for the elements of $\QQQ{g}$ we have been discussing.

\begin{definition}
For $v,w_1,w_2 \in H_L$ such that $\{w_1,w_2\}$ is isotropic and unimodular, define $X(v,w_1,w_2) = \eta(\CLL{x} \otimes \CL{y,z}{v}) \in \QQQ{g}$,
where $x,y,z \in \pi_1(\Sigma_g)$ have the same oriented intersection pattern as the curves in Figure \ref{figure:xwelldefined}.a and
$w_1 = \CAB{y}$ and $w_2 = \CAB{z}$.
\end{definition}

The paragraph before the definition shows that we can find appropriate $\{x,y,z\}$, and Lemma \ref{lemma:xwelldefined} implies
that $X(v,w_1,w_2)$ only depends on $\{v,w_1,w_2\}$.

\subsection{The image of $X(v,w_1,w_2)$ under $\psi$}
\label{section:xpsiimage}

Recall that $\psi$ is the natural map $\QQQ{g} \rightarrow \III{g}$ induced by the Reidemeister pairing $\REID{\cdot}{\cdot}$
from Lemma \ref{lemma:higherpairing}.  We now prove the following.

\begin{lemma}
\label{lemma:xpsiimage}
Consider $v,w_1,w_2 \in H_L$ such that $\{w_1,w_2\}$ is isotropic and unimodular.  Then
$\psi(X(v,w_1,w_2)) = \KRHO{v} - \KRHO{v+w_1} - \KRHO{v+w_2} + \KRHO{v+w_1+w_2}$.
\end{lemma}
\begin{proof}
Let $x,y,z \in \pi_1(\Sigma_g)$ be curves with the same oriented intersection pattern as the curves in Figure \ref{figure:xwelldefined}.a
such that $\CAB{y} = w_1$ and $\CAB{z} = w_2$.  We then have $X(v,w_1,w_2) = \eta(\CLL{x} \otimes \CL{y,z}{v})$, and the lemma
is equivalent to proving that
$$\REID{\CLL{x}}{\CL{y,z}{v}} = \KRHO{v} - \KRHO{v+\CAB{y}} - \KRHO{v+\CAB{z}} + \KRHO{v+\CAB{y}+\CAB{z}}.$$
Let $\rho : \CVR{g} \rightarrow \Sigma_g$ be the cover corresponding to $\KKK{g}$.  The group
of deck transformations is thus $H_L$.  Let $B_1 \subset \Sigma_g$ (resp. $B_2 \subset \Sigma_g)$)
be the one-holed torus on the left (resp.\ right) side of Figure \ref{figure:xwelldefined}.a.  Also, let
$A \subset \Sigma_g$ be the four-holed sphere ``between'' $B_1$ and $B_2$ in Figure \ref{figure:xwelldefined}.a.
We then have the following.
\begin{itemize}
\item $\rho^{-1}(A)$ is the disjoint union of $|H_L| = L^{2g}$ four-holed spheres each of which
projects homeomorphically onto $A$.
\item $\rho^{-1}(B_i)$ is the disjoint union of $|H_L|/L^2 = L^{2g-2}$ components.  If $\tilde{B}_i$
is one of those components, then $\tilde{B}_i$ is an $L^2$-holed torus and
$\rho|_{\tilde{B}_i} : \tilde{B}_i \rightarrow B_i$ is a cover with deck group $(\Z/L)^2$.
\end{itemize}
The homology class $\CL{y,z}{v}$ on $\CVR{g}$ can be realized by a simple closed curve $\gamma$ as
in Figure \ref{figure:calculate}.  If $\tilde{A}$
is the component of $\rho^{-1}(A)$ containing the basepoint, then this simple closed curve does the following.
\begin{itemize}
\item Beginning in the component $\KRHO{v} \cdot \tilde{A}$ of $\rho^{-1}(A)$, it goes through
a component of $\rho^{-1}(B_2)$ to arrive in $\KRHO{v+\CAB{y}} \cdot \tilde{A}$.
\item It then goes through a component of $\rho^{-1}(B_1)$ to arrive in
$\KRHO{v+\CAB{y}+\CAB{z}} \cdot \tilde{A}$.
\item It then goes through a component of $\rho^{-1}(B_2)$ to arrive in
$\KRHO{v+\CAB{z}} \cdot \tilde{A}$.
\item It finally goes through a component of $\rho^{-1}(B_1)$ to arrive back in $\KRHO{v} \cdot \tilde{A}$.
\end{itemize}
Let $\tilde{x}$ be the lift of $x$ contained in $\tilde{A}$.
As is evident from Figure \ref{figure:calculate}, the curve $\gamma$ intersects four different $H_L$-translates
of $\tilde{x}$, two with positive sign and two with negative sign.  Examining these intersections,
we see that
$$\REID{\CLL{x}}{\CL{y,z}{v}} = \KRHO{v} - \KRHO{v+\CAB{y}} - \KRHO{v+\CAB{z}} + \KRHO{v+\CAB{y}+\CAB{z}},$$
as desired.
\end{proof}

\Figure{figure:calculate}{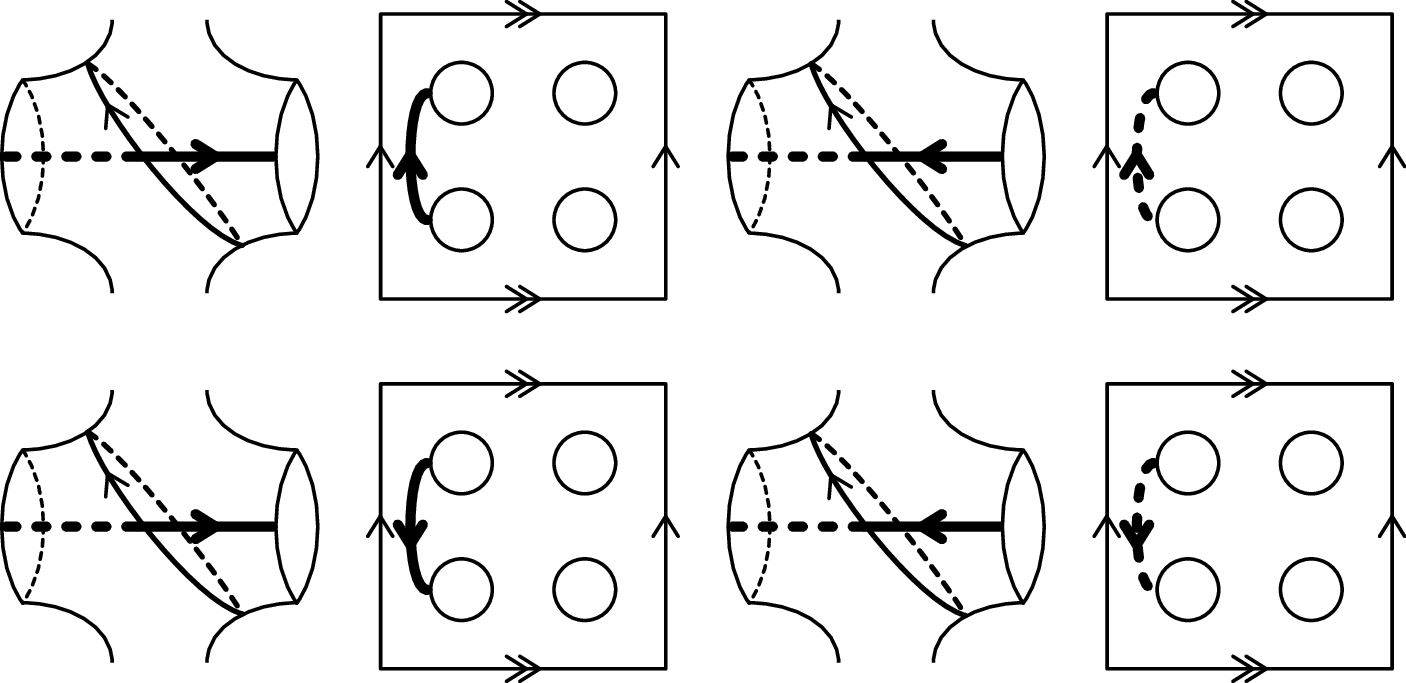}{The curve $\gamma$ in $\CVR{g}$ whose homology class
is $\CL{y,z}{v}$.  The dark portions are the lifts of $y^{\pm 1}$ and the dashed portions
are the lifts of $z^{\pm 1}$.  The numbers at the end of each segment indicate where the curve
goes next.  The 4-holed spheres are components of $\rho^{-1}(A)$ and the $L^2$-holed
tori (depicted here for $L=2$; the sides of the squares should be glued up in the indicated ways)
are components of $\rho^{-1}(B_1)$ and $\rho^{-1}(B_2)$.  The curve $\gamma$ intersects four
lifts $\tilde{x}_1,\ldots,\tilde{x}_4$ of $\rho^{-1}(x)$.  If $\tilde{x}$ is the lift of $x$
starting at the basepoint, then $\tilde{x}_1 = \KRHO{v} \cdot \tilde{x}$ and 
$\tilde{x}_2 = \KRHO{v + \CAB{y}} \cdot \tilde{x}$ and 
$\tilde{x}_3 = \KRHO{v + \CAB{y} + \CAB{z}} \cdot \tilde{x}$ and $\tilde{x}_4 = \KRHO{v+\CAB{z}} \cdot \tilde{x}$}

\subsection{$\QQQ{g}$ is spanned by the $X(v,w_1,w_2)$}
\label{section:qgenerators1}

In this section, we prove that the $X(v,w_1,w_2)$ span $\QQQ{g}$.  The proof will use the following lemma.

\begin{lemma}[{\cite[Lemma A.1]{PutmanCutPaste}}]
\label{lemma:commutatorgenerators}
Fix $g \geq 1$ and set $\pi = \pi_1(\Sigma_g)$.  The group $[\pi,\pi]$ is then generated by the set
$$\{\text{$\gamma$ $|$ $\gamma \in \pi$ can be realized by a simple closed separating curve}\}.$$
\end{lemma}

We now prove our lemma.

\begin{lemma}
\label{lemma:qgenerators1}
For $g \geq 1$, the vector space $\QQQ{g}$ is spanned by the set
$$\{\text{$X(v,w_1,w_2)$ $|$ $v,w_1,w_2 \in H_L$, $\{w_1,w_2\}$ is isotropic and unimodular}\}.$$
\end{lemma}
\begin{proof}
Define
$$\CCC{g}(\Z) = \Ker(\HH_1(\KKK{g};\Z) \longrightarrow \HH_1(\Sigma_g;\Z)).$$
Our proof will have three steps.  As notation, for $q,q' \in \QQQ{g}$,
write $q \equiv q'$ if $q$ and $q'$ are equal modulo $\eta(\CCC{g}(\Z) \otimes \CCC{g})$.

\BeginSteps
\begin{step}
Let $x,y,z \in \pi_1(\Sigma_g)$ be such that $z$ can be realized by a simple closed nonseparating curve
and $\{z\}$ is strongly essentially separate from $\{x,y\}$.  Also, let $v \in H_L$.  
Then $\eta(\CLL{x^L} \otimes \CL{y,z}{v}) \equiv 0$.
\end{step}

\Figure{figure:xgenerators}{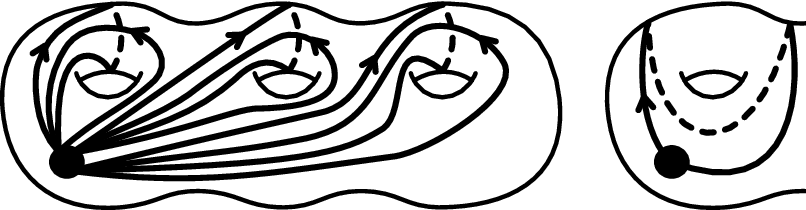}{a. A standard basis for $\pi_1(\Sigma_g)$
\CaptionSpace b. The curve $[\alpha_1,\alpha_2]$.}

Since $\Mod_g^1(L)$ contains all inner automorphisms
of $\pi_1(\Sigma_g)$, we can use the $\Mod_g^1(L)$-invariance of $\QQQ{g}$ to deduce that
$$\eta(\CLL{x^L} \otimes \CL{y,z}{v+j \cdot \CAB{z}}) = \eta(\CLL{z^{-j} x^L z^{j}} \otimes \CL{y,z}{v}) \equiv \eta(\CLL{x^L} \otimes \CL{y,z}{v}).$$
for all $j \in \Z$.  Consequently,
$$\eta(\CLL{x^L} \otimes \CL{y,z}{v}) \equiv \frac{1}{L} \sum_{j=0}^{L-1} \eta(\CLL{x}^L \otimes \CL{y,z}{v+j \cdot \CAB{z}}) = \frac{1}{L} \eta(\CLL{x^L} \otimes \CL{y,z^L}{v}) = 0.$$
This last equality follows from the fact that $(\CLL{x^L}, \CL{y,z^L}{v}) \in \SSSS{2}$

\begin{step}
$\eta(\CCC{g}(\Z) \otimes \CCC{g}) = \QQQ{g}$.
\end{step}

Our goal is to show that $q \equiv 0$ for all $q \in \QQQ{g}$.  To do this, it is enough to show
that $\eta(\CLL{\alpha_1^L} \otimes c) \equiv 0$ for all $\alpha_1 \in \pi_1(\Sigma_g)$ that can be realized by
a simple closed nonseparating curve and all $c \in \CCC{g}$.
Extend $\alpha_1$ to a standard basis $\{\alpha_1,\ldots,\alpha_{2g}\}$ for $\pi_1(\Sigma_g)$ as
in Figure \ref{figure:xgenerators}.a.  The vector space $\CCC{g}$ is spanned by the set
$$S=\{\text{$\CL{\alpha_i,\alpha_j}{v}$ $|$ $1 \leq i < j \leq 2g$, $v \in H_L$}\}.$$
Consider $\CL{\alpha_i,\alpha_j}{v} \in S$.  If $i \geq 3$ or $(i,j)=(1,2)$, then $[\alpha_i,\alpha_j]$ is
essentially separate from $\alpha_1^L$ (see Figure \ref{figure:xgenerators}.b for $(i,j)=(1,2)$) 
and thus $\eta(\CLL{x^L} \otimes \CL{\alpha_i,\alpha_j}{v}) = 0$.  Otherwise, $i \in \{1,2\}$
and $j > 2$.  It follows that $\{\alpha_j\}$ is strongly essentially separate from
$\{\alpha_1,\alpha_i\}$, and thus Step 1 implies 
$\eta(\CLL{\alpha_1^L} \otimes \CL{\alpha_i,\alpha_j}{v}) \equiv 0$. 

\begin{step}
$\eta(\CCC{g}(\Z) \otimes \CCC{g})$ is generated by the set
$$\{\text{$X(v,w_1,w_2)$ $|$ $v,w_1,w_2 \in H_L$, $\{w_1,w_2\}$ is isotropic and unimodular}\}.$$
\end{step}

\noindent
Set
$$S = \{\text{$\gamma \in \pi_1(\Sigma_g)$ $|$ $\gamma \neq 1$, $\gamma$ can be realized by a simple closed separating curve}\}.$$
Lemma \ref{lemma:commutatorgenerators} implies that $\CCC{g}(\Z)$ is generated by $\{\text{$\CLL{\gamma}$ $|$ $\gamma \in S$}\}$.  Consider
$\gamma \in S$.  Let $X_1$ and $X_2$ be the two surfaces into which $\gamma$ cuts $\Sigma_g$.  Order them so that
$X_1$ lies to the right of $\gamma$ and $X_2$ to the left.  We can then find a basis $B_{\gamma}^1 \cup B_{\gamma}^2$
for $\pi_1(\Sigma_g)$ with the following properties.
\begin{itemize}
\item For $\delta \in B_{\gamma}^i$, we have $\delta \in \Image(\pi_1(X_i) \rightarrow \pi_1(\Sigma_g))$.
\item Consider $\delta_1 \in B_{\gamma}^1$ and $\delta_2 \in B_{\gamma}^2$.  The
curves $\{\gamma,\delta_1,\delta_2\}$ then have the same oriented intersection pattern as the curves in Figure \ref{figure:xwelldefined}.a.
\end{itemize}
It follows that $\CCC{g}$ is spanned by the set $U_{\gamma} \cup V_{\gamma}$, where
$$U_{\gamma} = \{\text{$\CL{\delta_1,\delta_2}{v}$ $|$ $\delta_i \in B_{\gamma}^i$, $v \in H_L$}\} \quad \text{and} \quad V_{\gamma} = \{\text{$\CL{\delta,\delta'}{v}$ $|$ there exists $i$ such that $\delta,\delta' \in B_{\gamma}^i$, $v \in H_L$}\}.$$
We then have that $\CCC{g}(\Z) \otimes \CCC{g}$ is spanned by the set $Z \cup W$, where
$$Z = \{\text{$\CLL{\gamma} \otimes c$ $|$ $\gamma \in S$, $c \in U_{\gamma}$}\} \quad \text{and} \quad W = \{\text{$\CLL{\gamma} \otimes c$ $|$ $\gamma \in S$, $c \in V_{\gamma}$}\}.$$
For $\CLL{\gamma} \otimes \CL{\delta_1,\delta_2}{v} \in Z$, we have $\eta(\CLL{\gamma} \otimes \CL{\delta_1,\delta_2}{v}) = X(v,\CAB{\delta}_1,\CAB{\delta}_2)$.  For
$\CLL{\gamma} \otimes \CL{\delta,\delta'}{v} \in W$, the curves $\gamma$ and $[\delta,\delta']$ are essentially separate, so
$\eta(\CLL{\gamma} \otimes \CL{\delta,\delta'}{v})=0$.  The desired result follows.
\end{proof}

\subsection{Relations between the $X(w,v_1,v_2)$}
\label{section:qrelations}

The goal of this section is to prove the following lemma, which gives relations
between the $X(v,w_1,w_2)$.

\begin{lemma}
\label{lemma:qrelations}
Let $\{w_1,w_2\} \subset \HH_1(\Sigma_{g};\Z/L)$ be an isotropic and unimodular set.  Then the
following hold for all $v \in \HH_1(\Sigma_g;\Z/L)$.
\begin{enumerate}
\item $X(v,w_1,w_2) = X(v,w_2,w_1)$
\item $X(v,-w_1,w_2) = -X(v-w_1,w_1,w_2)$
\item $\sum_{i=0}^{L-1} X(v+i \cdot w_1,w_1,w_2) = 0$ and $\sum_{i=0}^{L-1} X(v+i \cdot w_2,w_1,w_2) = 0$
\item Let $w_3 \in \HH_1(\Sigma_g;\Z/L)$ be such that $\{w_1,w_2,w_3\}$ is unimodular, such
that $i(w_2,w_3)=0$, and such that $-1 \leq i(w_1,w_3) \leq 1$.  Then
$X(v,w_1+w_3,w_2) = X(v,w_1,w_2) + X(v+w_1,w_3,w_2)$.
\end{enumerate}
\end{lemma}

\begin{remark}
It is instructive to check (using the formula in Lemma \ref{lemma:xpsiimage}) that each of these
relations is taken to $0$ by $\psi$.
\end{remark}

\begin{proof}[{Proof of Lemma \ref{lemma:qrelations}}]
Let $x,y,z \in \pi_1(\Sigma_g)$ be curves such that $\{x,y,z\}$ has the same oriented intersection pattern as the curves
in Figure \ref{figure:xwelldefined}.a and such that $\CAB{y} = w_1$ and $\CAB{z} = w_2$.  Hence
$X(v,w_1,w_2) = \eta(\CLL{x} \otimes \CL{y,z}{v})$.

For item 1, observe that if we flip
$y$ and $z$, then our curves no longer have the same oriented intersection pattern as the curves
in Figure \ref{figure:xwelldefined}.a (they do have the same unoriented intersection pattern).  To restore the correct orientations, we must
reverse $x$.  In other words,
$$X(v,w_2,w_1) = \eta(\CLL{x^{-1}} \otimes \CL{z,y}{v}) = \eta((-\CLL{x}) \otimes (-\CL{y,z}{v})) = \eta(\CLL{x} \otimes \CL{y,z}{v}) = X(v,w_1,w_2),$$
as desired.

For item 2, observe that the set of curves $\{x,y^{-1} x, z\}$ has the same oriented intersection pattern
as $\{x,y,z\}$ (see the example in \S \ref{section:intersectionpattern}).  
Also, since $\CAB{x}=0$, we have $\CAB{y^{-1} x} = -\CAB{y}=-w_1$.  Thus
we can apply Lemma \ref{lemma:commutatoridentities} to get that
\begin{align*}
X(v,-w_1,w_2) &= \eta(\CLL{x} \otimes \CL{y^{-1} x,z}{v}) = \eta(\CLL{x} \otimes \CL{y^{-1},z}{v}) + \eta(\CLL{x} \otimes \CL{x,z}{v-w_1})\\
& = -\eta(\CLL{x} \otimes \CL{y,z}{v-w_1}) + \eta(\CLL{x} \otimes \CL{x}{v-w_1}) - \eta(\CLL{x} \otimes \CL{x}{v-w_1+w_2}).
\end{align*}
Since $x$ is essentially separate from itself, the last two terms vanish and this equals
$$-\eta(\CLL{x} \otimes \CL{y,z}{v-w_1}) = -X(v-w_1,w_1,w_2),$$
as desired.

For item 3, item 1 implies that it is enough to prove that $\sum_{i=0}^{L-1} X(v+i \cdot w_2,w_1,w_2) = 0$.
Observe that
$(\CLL{x}, \CL{y,z^L}{v}) \in \SSSS{2}$, so by Lemma \ref{lemma:commutatoridentities} we have
$$0 = \eta(\CLL{x} \otimes \CL{y,z^L}{v}) = \sum_{i=0}^{L-1} \eta(\CLL{x} \otimes \CL{y,z}{v+i \cdot \CAB{z}}) = \sum_{i=0}^{L-1} X(v+i \cdot w_2,w_1,w_2),$$
as desired.

\Figure{figure:xrelations}{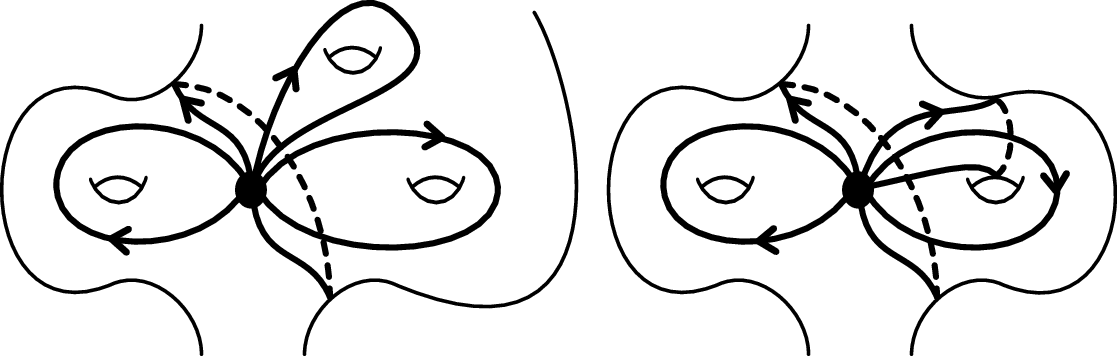}{a. $i(y,y')=0$ and the sets of curves $\{x,y,z\}$ and
$\{x,y',z\}$ and $\{x,yy',z\}$ have the same oriented intersection pattern as the curves in
Figure \ref{figure:xwelldefined}.a
\CaptionSpace b. $i(y,y')=-1$ and the sets of curves $\{x,y,z\}$ and
$\{x,y',z\}$ and $\{x,yy',z\}$ and $\{x,y(y')^{-1},z\}$ have the same oriented intersection pattern
as the curves in Figure \ref{figure:xwelldefined}.a}

We conclude with item 4.  Here we will have to change our curves $x$ and $y$ and $z$.  
We will prove shortly that we can find $x,y,y',z \in \pi_1(\Sigma_g)$ with the following
properties.
\begin{itemize}
\item $\CAB{y} = w_1$ and $\CAB{y}' = w_3$ and $\CAB{z} = w_2$.
\item The sets of curves $\{x,y,z\}$ and $\{x,y',z\}$ and $\{x,yy',z\}$ each have the same
oriented intersection pattern as the curves in Figure \ref{figure:xwelldefined}.a.
\end{itemize}
See Figures \ref{figure:xrelations}.a,b.
Assuming this for the moment, the proof is completed by appealing to Lemma
\ref{lemma:commutatoridentities} to deduce that
\begin{align*}
X(v,w_1+w_3,w_2) &= \eta(\CLL{x} \otimes \CL{yy',z}{v}) = \eta(\CLL{x} \otimes \CL{y,z}{v}) + \eta(\CLL{x} \otimes \CL{y',z}{v+\CAB{y}})\\
&= X(v,w_1,w_2) + X(v+w_1,w_3,w_2).
\end{align*}
It remains to prove the above claim.  By Lemma \ref{lemma:realizehomology}, we can find unbased
oriented simple closed curves $Y$ and $Y'$ and $Z$ on $\Sigma_g$ with the following properties.
\begin{itemize}
\item The $\Z/L$-homology classes of $Y$, $Y'$, and $Z$ are $w_1$, $w_3$, and $w_2$, respectively.
\item $Z$ is disjoint from $Y$ and $Y'$.  Also, $Y$ and $Y'$ are disjoint if $i(w_1,w_3)=0$ and intersect
once if $i(w_1,w_3)=\pm 1$.
\end{itemize}
We can then connect $Y$, $Y'$, and $Z$ to the basepoint to obtain curves $y,y',z \in \pi_1(\Sigma_g)$
such that $\{y,y',z\}$ has the same oriented intersection pattern as the curves in Figure
\ref{figure:xrelations}.a (if $i(w_1,w_3)=0$) or \ref{figure:xrelations}.b (if $i(w_1,w_3)=-1$) or
\ref{figure:xrelations}.b with the orientation of the curve $y'$ reversed (if $i(w_1,w_3)=1$).  It
is then clear that we can find $x \in \pi_1(\Sigma_g)$ such that $\{x,y,y',z\}$ has the indicated
oriented intersection pattern, as desired.
\end{proof}

\section{The map $\psi : \QQQ{g} \rightarrow \III{g}$ is an isomorphism}
\label{section:linearalgebra}

The purpose of this section is to prove that the map $\psi : \QQQ{g} \rightarrow \III{g}$
is an isomorphism.  The actual proof is in \S \ref{section:psiisomorphism}.  This
is proceeded by \S \ref{section:qgenerators2}, which constructs a generating set $V$ for $\QQQ{g}$
that is slightly smaller than the generating set determined in \S \ref{section:qgenerators1}.

Throughout this section, we will freely use the main results of \S \ref{section:generatorsandrelations} (i.e.\ Lemmas \ref{lemma:xpsiimage} and
\ref{lemma:qgenerators1} and \ref{lemma:qrelations}).

\subsection{A smaller generating set}
\label{section:qgenerators2}

Fix a symplectic basis $B=\{a_1,b_1,\ldots,a_{g},b_{g}\}$ for $H_L$.
Define $V = V_1 \cup V_2 \subset \QQQ{g}$, where the $V_i$ are as follows.
\begin{align*}
\begin{split}
V_1 = \{\text{$X(v,s_1,s_2)$ $|$ }&\text{$v \in H_L$, $s_1,s_2 \in B$ distinct, $i(s_1,s_2)=0$}\},\\
\end{split}\\
\begin{split}
V_2 = \{\text{$X(v,s_1,s_1 + e s_2)$ $|$ }&\text{$v \in H_L)$, $s_1,s_2 \in B$ distinct, $e \in \{-1,1\}$, $i(s_1,s_2)=0$}\},\\
\end{split}
\end{align*}
The goal of this section is to prove the following.

\begin{lemma}
\label{lemma:qgenerators2}
For $g \geq 4$, the vector space $\QQQ{g}$ is spanned by $V$.
\end{lemma}

We begin with the following relations in $\QQQ{g}$.

\begin{lemma}
\label{lemma:newrelation1}
Fix $g \geq 1$, and let $\{s_1,s_2,s_3\} \subset H_L$ be a unimodular set
such that $i(s_1,s_3)=i(s_2,s_3)=0$ and $-1 \leq i(s_1,s_2) \leq 1$.
Then for all $v \in H_L$ we have the following two relations.
\begin{align*}
X(v,s_2,s_3) &= X(v+s_1,s_2,s_3) + X(v,s_1,s_3) - X(v+s_2,s_1,s_3), \\
X(v-s_1,s_2,s_3) &= X(v,s_2,s_3) + X(v-s_1,s_1,s_3) - X(v-s_1+s_2,s_1,s_3).
\end{align*}
\end{lemma}
\begin{proof}
The first relation follows from the fact that
$$X(v,s_1+s_2,s_3) = X(v,s_1,s_3) + X(v+s_1,s_2,s_3) \quad \text{and} \quad 
X(v,s_1+s_2,s_3) = X(v,s_2,s_3) + X(v+s_2,s_1,s_3).$$
The second follows from the first via the substitution $v \mapsto v-s_1$.
\end{proof}

We next show that $\Span{V}$ contains several other classes of elements.  Define
\begin{align*}
\begin{split}
V_3 = \{\text{$X(v,s_1 + e s_2,s_3 + e' s_4)$ $|$ }&\text{$v \in H_L$, $s_1,s_2,s_3,s_4 \in B$ distinct, 
$e,e' \in \{-1,1\}$,}\\
&\text{$i(s_1,s_3)=\pm 1$, $i(s_1 + e s_2,s_3 + e' s_4)=0$}\},\\
\end{split}\\
\begin{split}
V_4 = \{\text{$X(v,s_1+e s_2, s_1 + e s_2 +e' s_3)$ $|$ }&\text{$v \in H_L$, $s_1,s_2,s_3 \in B$ distinct, $e,e' \in \{-1,1\}$,}\\
&\text{$i(s_1,s_2) = \pm 1$, $i(s_1,s_3)=i(s_2,s_3)=0$}\}.\\
\end{split}
\end{align*}
We then have the following.

\begin{lemma}
\label{lemma:eliminatev34}
For $g \geq 1$, we have $V_3,V_4 \subset \Span{V}$.
\end{lemma}
\begin{proof}
For $x,y \in \QQQ{g}$, write $x \equiv y$ if $x$ and $y$ are equal modulo $\Span{V_1,V_2}$.
First consider $X(v,w_1,w_2) \in V_3$.  Our goal is to show that $X(v,w_1,w_2) \equiv 0$.  For
concreteness, we will do the case $X(v,w_1,w_2) = X(v,a_i + b_j, b_i + a_j)$ for
some $1 \leq i,j \leq g$ with $i \neq j$; the other cases are similar.

Observe first that Lemma \ref{lemma:newrelation1} implies that
$$X(v,b_i+a_j,a_i+b_j) = X(v+a_i,b_i+a_j,a_i+b_j) + X(v,a_i,a_i+b_j) - X(v+b_i+a_j,a_i,a_i+b_j).$$
Since $X(v,a_i,a_i+b_j), X(v+b_i+a_j,a_i,a_i+b_j) \in V_2$, we deduce
that $X(v,a_i + b_j, b_i + a_j) \equiv X(v+a_i,a_i+b_j, b_i+a_j)$.  In a similar manner, we
have $X(v+a_i,a_i+b_j,b_i+a_j) \equiv X(v+(a_i+b_j),a_i+b_j,b_i+a_j)$.  Iterating
this, we obtain $X(v,a_i+b_j,b_i+a_j) \equiv X(v + k(a_i+b_j),a_i+b_j,b_i+a_j)$ for all
$k \in \Z$.  But this implies that
$$X(v,a_i+b_j,b_i+a_j) \equiv \frac{1}{L} \sum_{k=0}^{L-1} X(v + k(a_i+b_j), a_i+b_j, b_i+a_j) = 0,$$
as desired.

Now consider $X(v',w_1',w_2') \in V_4$.  We will show that $X(v',w_1',w_2')$ can be written as a
a linear combination of elements of $V_1 \cup V_2 \cup V_3$.  For concreteness, we will do the case
$X(v',w_1',w_2') = X(v,a_i+b_i,a_i+b_i+a_j)$ for some $1 \leq i,j \leq g$ with $i \neq j$; the
other cases are similar.  In this case, we have
\begin{align*}
X(v,a_i+b_i,a_i+b_i+a_j) = &X(v,(a_i+b_j)+(b_i-b_j),a_i+b_i+a_j)\\
= &X(v,a_i+b_j,a_i+b_i+a_j) + X(v+a_i+b_j,b_i-b_j,a_i+b_i+a_j)\\
= &X(v,a_i+b_j,a_i) + X(v+a_i,a_i+b_j,b_i+a_j) \\
&+ X(v+a_i+b_j,b_i-b_j,b_i) + X(v+a_i+b_j+b_i,b_i-b_j,a_i+a_j),
\end{align*}
as desired.
\end{proof}

Next, say that $v \in H_L$ is a {\em simple element of length at most $k$} if
it can be written as
$$v = \sum_{i=1}^g (c_i a_i + d_i b_i)$$
for some $c_i,d_i \in \{-1,0,1\}$ such that at most $k$ of the $c_i$ and $d_i$ are nonzero.  Define
\begin{align*}
V' = \{\text{$X(v,w_1,w_2)$ $|$ }&\text{$v,w_1,w_1 \in H_L$, $\{w_1,w_2\}$ is
isotropic and unimodular, and}\\
&\text{$w_1$ is a simple element of length at most $3$}
\end{align*}
We then have the following

\begin{lemma}
\label{lemma:eliminatevprime}
For $g \geq 1$, we have $V' \subset \Span{V}$.
\end{lemma}
\begin{proof}
An easy case-by-case check shows that one can use the ``bilinearity relations'' (relations 1, 2, and
4 in Lemma \ref{lemma:qrelations}) to express every element of $V'$ as a linear combination
of elements of $V \cup V_3 \cup V_4$, and thus via Lemma \ref{lemma:eliminatev34} as a linear combination
of elements of $V$.
\end{proof}

We finally prove Lemma \ref{lemma:qgenerators2}.

\Figure{figure:modgenerators}{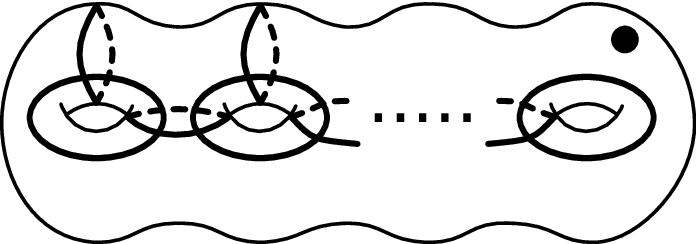}{Generators for $\Mod_g^1$}

\begin{proof}[{Proof of Lemma \ref{lemma:qgenerators2}}]
The mapping class group $\Mod_{g}^1$ acts on both $H_L$ and $\QQQ{g}$.  The action on $H_L$
is transitive on pairs $\{w_1,w_2\}$ of vectors that are isotropic and unimodular.  It follows that the
$\Mod_g^1$-orbit of the set $V$ contains every element $X(v,w_1,w_2)$.  Lemma
\ref{lemma:qgenerators1} says that these generate $\QQQ{g}$, so we conclude that it is enough
to show that $\Span{V}$ is invariant under $\Mod_g^1$.  Let $\{\delta_1,\ldots,\delta_{2g+1}\}$
be the simple closed curves depicted in Figure \ref{figure:modgenerators}.  The set
$$S = \{\text{$T_{\delta_i}$ $|$ $1 \leq i \leq 2g+1$}\}$$
then generates $\Mod_g^1$ (see \cite{FarbMargalitPrimer}).  It is enough to prove that
$s \cdot X(v,w_1,w_2) \in \Span{V}$ for $s \in S^{\pm 1}$ and $X(v,w_1,w_2) \in V$.  However,
it is easy to see that $s(w_1)$ is a simple element of length at most $3$, so
$s \cdot X(v,w_1,w_2) \in V'$ and the desired result follows from Lemma \ref{lemma:eliminatevprime}.
\end{proof}

\subsection{$\psi$ is an isomorphism}
\label{section:psiisomorphism}

In this section, we prove Lemma \ref{lemma:psiisomorphism}, which we recall asserts that 
for $g \geq 3$, the map $\psi : \QQQ{g} \rightarrow \III{g}$ is an isomorphism.
Our proof is lengthy, but the basic idea is as follows.
\begin{itemize}
\item Using the relations in $\QQQ{g}$, we will show that the set
$\Span{V}$ is generated by a set containing $\Dim{\Q[H_L]}-1$ elements.
\item By carefully examining the image of $\psi$, we will show that
$\Q[H_L] = \psi(\Span{V}) + \Span{\KRHO{0}}$.
\end{itemize}
A simple dimension count will then establish the lemma.

Some parts of our proof will be by induction on $g$.  We will thus need notation for $\Q[H_L]$ which
takes $g$ into account, so define $\BBB{g} = \Q[\HH_1(\Sigma_g;\Z/L)]$.

To make the calculations a bit more palatable, we will break this down into several steps.
We first determine what $\psi$ does to $V_1$.

\begin{lemma}
\label{lemma:v1injective}
Fix $g \geq 1$.
Set $\BBB{g}^1 = \Span{\{\text{$\KRHO{c a_i + d b_i}$ $|$ $c,d \in \Z/L$, $1 \leq i \leq g$}\}}$.
Then the map $\psi|_{\Span{V_1}}$ is injective and
$\BBB{g} = \psi(\Span{V_1}) \oplus \BBB{g}^1$.
\end{lemma}
\begin{proof}
The proof will by induction on $g$.  For the base case $g=1$, the set $V_1$ is empty
and the assertion is trivial.  Assume now that $g \geq 2$ and that the lemma is true
for all smaller $g$.  Define
\begin{align*}
V_1^I &= \{\text{$X(v,s_1,s_2) \in V_1$ $|$ $s_1,s_2 \notin \{a_{g},b_{g}\}$}\},\\
V_1^A &= \{\text{$X(v,a_g,s)$ $|$ $X(v,a_g,s) \in V_1$}\},\\
V_1^B &= \{\text{$X(v,b_g,s)$ $|$ $X(v,b_g,s) \in V_1$}\},
\end{align*}
so $V = V_1^I \cup V_1^A \cup V_1^B$.  The proof will consist of three steps.

\BeginSteps
\begin{step}
Set $\BBB{g}^2 = \Span{\{\text{$\KRHO{c a_i + d b_i + e a_{g} + f b_{g}}$ $|$ $c,d,e,f \in \Z/L$, $1 \leq i \leq g-1$}\}}$.
Then the map $\psi|_{\Span{V_1^I}}$ is injective and
$\BBB{g} = \psi(\Span{V_1^I}) \oplus \BBB{g}^2$.
\end{step}

\noindent
For $e,f \in \Z/L$, define
\begin{align*}
\BBB{g}(e,f) &= \Span{\{\text{$\KRHO{v + e a_{g} + f b_{g}}$ $|$ $v \in \Span{a_1,b_1,\ldots,a_{g-1},b_{g-1}}$}\}},\\
\BBB{g}^2(e,f) &= \Span{\{\text{$\KRHO{c a_i + d b_i + e a_{g} + f b_{g}}$ $|$ $c,d \in \Z/L$, $1 \leq i \leq g-1$}\}},\\
V_1^{I}(e,f) &= \{\text{$X(v + e a_{g} + f b_{g},s_1,s_2) \in V_1$ $|$ $s_1,s_2 \notin \{a_g,b_g\}$, $v \in \Span{a_1,b_1,\ldots,a_{g-1},b_{g-1}}$}\}.
\end{align*}
Observe that
$$\BBB{g} = \bigoplus_{e,f \in \Z/L} \BBB{g}(e,f) \quad \text{and} \quad 
\BBB{g}^2 = \bigoplus_{e,f \in \Z/L} \BBB{g}^2(e,f) \quad \text{and} \quad
V_1^I = \bigsqcup_{e,f \in \Z/L} V_1^{I}(e,f).$$
Moreover, for all $e,f \in \Z/L$ we have $\psi(V_1^{I}(e,f)) \subset \BBB{g}(e,f)$.  We conclude that it
is enough to prove that $\psi|_{\psi(\Span{V_1^{I}(e,f)})}$ is injective and
$\BBB{g}(e,f) = \psi(\Span{V_1^{I}(e,f)}) \oplus \BBB{g}^2(e,f)$
for all $e,f \in \Z/L$.

Consider $e,f \in \Z/L$.  The map $\KRHO{v} \mapsto \KRHO{v + e a_{g} + f b_{g}}$ induces an isomorphism
$\rho:\BBB{g-1} \rightarrow \BBB{g}(e,f)$ that restricts to an isomorphism $\BBB{g-1}^1 \cong \BBB{g}^2(e,f)$.
Let $V_{1,g-1} \subset \QQQ{g-1}$ be the
$(g-1)$-dimensional analogue of $V_1$ and $\psi_{g-1} : \QQQ{g-1} \rightarrow \BBB{g-1}$ be the $(g-1)$-dimensional
analogue of $\psi$.  The map $X(v,w_1,w_2) \mapsto X(v + e a_{g} + f b_{g},w_1,w_2)$
induces a homomorphism $\QQQ{g-1} \rightarrow \QQQ{g}$ that restricts to a
surjection $\rho' : \Span{V_{1,g-1}} \rightarrow \Span{V_1^{I}(e,f)}$.  Observe that the diagram
$$\begin{CD}
\Span{V_{1,g-1}} @>{\psi_{g-1}}>> \BBB{g-1}\\
@V{\rho'}VV                        @V{\rho}VV\\
\Span{V_1^{I}(e,f)} @>{\psi}>> \BBB{g}(e,f)
\end{CD}$$
commutes.  By the induction hypothesis, $\psi_{g-1}|_{\Span{V_{1,g-1}}}$ is injective and
$\BBB{g-1} = \psi_{g-1}(\Span{V_{1,g-1}}) \oplus \BBB{g-1}^1$.
We conclude that $\rho'$ is injective and thus an isomorphism.  Moreover, $\psi|_{\Span{V_1^{I}(e,f)}}$ is injective and
$\BBB{g}(e,f) = \psi(\Span{V_1^{I}(e,f)}) \oplus \BBB{g}^2(e,f)$,
as desired.

\begin{step}
Set
$\BBB{g}^3 = \Span{\{\text{$\KRHO{c a_i + d b_i + f b_{g}}$, $\KRHO{e a_{g} + f b_{g}}$ $|$ $c,d,e,f \in \Z/L$, $1 \leq i \leq g-1$}\}}$.
Then the map $\psi|_{\Span{V_1^I, V_1^A}}$ is injective and
$\BBB{g} = \psi(\Span{V_1^I, V_1^A}) \oplus \BBB{g}^3$.
\end{step}

\noindent
Define
$$V_1^{A,1} = \{\text{$X(c a_i + d b_i + e a_g + f b_g,a_g,s)$ $|$ $1 \leq i \leq g-1$, $c,d,e,f \in \Z/L$, $s \in \{a_i,b_i\}$}\} \subset V_1^A.$$
We claim that $\Span{V_1^I,V_1^A} = \Span{V_1^I,V_1^{A,1}}$.  Indeed, consider $X(w,a_g,s) \in V_1^A$.  By
Lemma \ref{lemma:newrelation1}, for any $s' \in \{a_1,b_1,\ldots,a_{g-1},b_{g-1}\}$ with $s' \neq s$ and $i(s,s')=0$, we
have
$$X(w-s',a_g,s) = X(w,a_g,s) + X(w-s',s',s) - X(w-s'+a_g,s',s).$$
Hence modulo $\Span{V_1^I}$, we have $X(w,a_g,s)$ equal to $X(w-s',a_g,s)$.  Iterating this, modulo $\Span{V_1^I}$
we have $X(w,a_g,s)$ equal to an element of $V_1^{A,1}$, as desired.

For $x \in \Z/L$, we will denote by $|x|$ the
unique integer representing $x$ with $0 \leq |x| < L-1$.
Noting that $\psi(V_1^{A,1}) \subset \BBB{g}^2$, we claim that $\BBB{g}^2 = \psi(\Span{V_1^{A,1}}) + \BBB{g}^3$.
Indeed, assume that there is some $\KRHO{c a_i + d b_i + e a_g + f b_g} \in \BBB{g}^2$ that is
not in $\psi(\Span{V_1^{A,1}}) + \BBB{g}^3$.
Choose $\KRHO{c a_i + d b_i + e a_g + f b_g}$ such that
$|c|+|d|+|e|$ is minimal among elements with this
property.  By assumption we must have $|e|$ and one of $|c|$ or $|d|$ (say $|c|$) nonzero.  Setting
$w' = c a_i + d b_i + e a_g + f b_g$, we then
have $X(w'-a_i-a_g,a_g,a_i) \in V_1^{A,1}$ and
$$\KRHO{w'} - \psi(X(w'-a_i-a_g,a_g,a_i)) = \KRHO{w'-a_i} + \KRHO{w'-a_g} - \KRHO{w'-a_i-a_g}.$$
We conclude that one of $\KRHO{w'-a_i}$, $\KRHO{w'-a_g}$, or $\KRHO{w'-a_i-a_g}$ is not in
$\psi(\Span{V_1^{A,1}}) + \BBB{g}^3$, contradicting the minimality of $|c|+|d|+|e|$.

Now define
\begin{align*}
V_1^{A,2} = &\{\text{$X(c a_i + d b_i + e a_g + f b_g,a_g,a_i)$ $|$ $1 \leq i \leq g-1$, $c,d,e,f \in \Z/L$}\} \\
&\quad \cup \{\text{$X(d b_i + e a_g + f b_g,a_g,b_i)$ $|$ $1 \leq i \leq g-1$, $d,e,f \in \Z/L$}\} \subset V_1^{A,1},\\
V_1^{A,3} = &\{\text{$X(v,a_g,s) \in V_1^{A,2}$ $|$ the $s$ and $a_g$-coordinates of $v$ do not equal $L-1$}\} \subset
V_1^{A,2}.
\end{align*}
We will prove that $\Span{V_1^{A,3}} = \Span{V_1^{A,1}}$.  Using the third relation in Lemma \ref{lemma:qrelations},
we see that $\Span{V_1^{A,3}} = \Span{V_1^{A,2}}$.  It is thus enough to prove that
$\Span{V_1^{A,1}} = \Span{V_1^{A,2}}$.  An element of
$V_1^{A,1} \setminus V_1^{A,2}$ is of the form $X(w'', a_g, b_i)$.  Lemma \ref{lemma:newrelation1} says that
$$X(w''-a_i,b_i,a_g) = X(w'',b_i,a_g) + X(w''-a_i,a_i,a_g) - X(w''-a_i+b_i,a_i,a_g).$$
Iterating this, we conclude that modulo $\Span{V_1^{A,2}}$, we have $X(w'',a_g,b_i)$ equal to an
element of $V_1^{A,2}$, as desired.

We deduce from the above two paragraphs
that $\BBB{g}^2 = \psi(\Span{V_1^{A,3}}) + \BBB{g}^3$.  Since $V_1^{A,3}$ contains
\begin{align*}
(g-1)(L^2(L-1)^2 + L(L-1)^2) &= ((g-1) (L^2-1) L^2 + L^2) - ((g-1)(L^2-1)L + L^2)\\
&= \Dim(\BBB{g}^2) - \Dim(\BBB{g}^3)
\end{align*}
elements, we obtain that $\psi|_{\Span{V_1^{A,3}}}$ is injective and
$\BBB{g}^2 = \psi(\Span{V_1^{A,3}}) \oplus \BBB{g}^3$.  By Step 1,
the fact that $\Span{V_1^I,V_1^A} = \Span{V_1^I,V_1^{A,1}}$, and the fact that
$\Span{V_1^{A,3}} = \Span{V_1^{A,1}}$,
we conclude that $\psi|_{\Span{V_1^I, V_1^{A}}}$
is injective and $\BBB{g} = \psi(\Span{V_1^I, V_1^{A}}) \oplus \BBB{g}^3$, as desired.

\begin{step}
Recall that $\BBB{g}^1 = \Span{\{\text{$\KRHO{c a_i + d b_i}$ $|$ $c,d \in \Z/L$, $1 \leq i \leq g$}\}}$.  The
map $\psi|_{\Span{V_1^I, V_1^A, V_1^B}}$ is injective and
$\BBB{g} = \psi(\Span{V_1^I, V_1^A, V_1^B}) \oplus \BBB{g}^1$.
\end{step}

\noindent
The argument for this step is very similar to the argument in Step 2, so we only sketch it.  Define
\begin{align*}
V_1^{B,1} = &\{\text{$X(c a_i + d b_i + f b_g,b_g,s)$ $|$ $1 \leq i \leq g-1$, $c,d,f \in \Z/L$, $s \in \{a_i,b_i\}$}\} \subset V_1^B,\\
V_1^{B,2} = &\{\text{$X(c a_i + d b_i + f b_g,b_g,a_i)$ $|$ $1 \leq i \leq g-1$, $c,d,f \in \Z/L$}\} \\
&\quad \cup \{\text{$X(d b_i + f b_g,b_g,b_i)$ $|$ $1 \leq i \leq g-1$, $d,f \in \Z/L$}\} \subset V_1^{B,1},\\
V_1^{B,3} = &\{\text{$X(v,b_g,s) \in V_1^{B,2}$ $|$ the $s$ and $b_g$-coordinates of $v$ do not equal $L-1$}\} \subset V_1^{B,2}.
\end{align*}
Noting that $\psi(V_1^{B,1}) \subset \BBB{g}^3$, arguments similar to those in Step 2
show that $\Span{V_1^I, V_1^A, V_1^B} = \Span{V_1^I, V_1^A, V_1^{B,1}}$, that
$\BBB{g}^3 = \psi(\Span{V_1^{B,1}}) + \BBB{g}^1$, that
$\Span{V_1^{B,1}} = \Span{V_1^{B,2}}$, and that
$\Span{V_1^{B,2}} = \Span{V_1^{B,3}}$.

We deduce that $\BBB{g}^3 = \psi(\Span{V_1^{B,3}}) + \BBB{g}^1$.  Since $V_1^{B,3}$ contains
\begin{align*}
(g-1)((L-1)^2 L+(L-1)^2) &= ((g-1)(L^2-1)L + L^2) - (g(L^2-1)+1)\\
&= \Dim(\BBB{g}^3) - \Dim(\BBB{g}^1)
\end{align*}
elements, we obtain that $\psi|_{\Span{V_1^{B,3}}}$ is injective and
$\BBB{g}^3 = \psi(\Span{V_1^{B,3}}) \oplus \BBB{g}^1$.  By Step 2 and
the identities
$$\Span{V_1^I,V_1^A,V_1^B} = \Span{V_1^I,V_1^{A},V_1^{B,1}} \quad \text{and} \quad \Span{V_1^{B,3}} = \Span{V_1^{B_1}},$$
we conclude that
$\psi|_{\Span{V_1^I,V_1^{A}, V_1^B}}$
is injective and $\BBB{g} = \psi(\Span{V_1^I, V_1^{A}, V_1^B}) \oplus \BBB{g}^1$, as desired.
\end{proof}

Our final lemma is a further relation in $\QQQ{g}$.

\begin{lemma}
\label{lemma:newrelation2}
Let $\{a_1',b_1',a_2',b_2'\}$ be a unimodular subset of $\HH_1(\Sigma_g;\Z/L)$ with
$i(a_1',b_1')=i(a_2',b_2')=1$ and $i(a_1',a_2')=i(a_1',b_2')=i(b_1',a_2')=i(b_1',b_2')=0$.  Then
for all $v \in \HH_1(\Sigma_g;\Z/L)$ we have
\begin{align*}
&X(v,a_1',a_2') - X(v+b_1',a_1',a_2') - X(v+b_2',a_1',a_2') + X(v+b_1'+b_2',a_1',a_2')\\
&\quad \quad = X(v,b_1',b_2') - X(v+a_1',b_1',b_2') - X(v+a_2',b_1',b_2') + X(v+a_1'+a_2',b_1',b_2')
\end{align*}
\end{lemma}
\begin{proof}
The group $\Sp_{2g}(\Z)$ acts on $\QQQ{g}$, and there exists some $f \in \Sp_{2g}(\Z)$ such that
$f(a_i') = a_i$ and $f(b_i') = b_i$ for $i=1,2$.  We can therefore assume that $a_i'=a_i$ and $b_i'=b_i$ for $i=1,2$.
But an easy calculation shows that $\psi$ takes both sides of our relation to the same element of $\BBB{g}$,
so the lemma follows from Lemma \ref{lemma:v1injective}.
\end{proof}

We can now prove Lemma \ref{lemma:psiisomorphism}.

\begin{proof}[{Proof of Lemma \ref{lemma:psiisomorphism}}]
Define $\QQQ{g}' = \QQQ{g} / \Span{V_1}$ and $\BBB{g}' = \BBB{g} / \psi(\Span{V_1})$.  We
have an induced map $\psi' : \QQQ{g}' \rightarrow \BBB{g}'$.  Using the direct sum
decomposition of Lemma \ref{lemma:v1injective}, we
will identify $\BBB{g}'$ with the subspace
$$\Span{\{\text{$\KRHO{c a_i + d b_i}$ $|$ $c,d \in \Z/L$, $1 \leq i \leq g$}\}}$$
of $\BBB{g}$.  Letting $V_2' \subset \QQQ{g}'$ be the image of $V_2 \subset \QQQ{g}$,
Lemmas \ref{lemma:qgenerators2} and \ref{lemma:v1injective} say
that it is enough to prove that $\psi'|_{\Span{V_2'}}$ is injective
and $\BBB{g}' = \psi'(\Span{V_2'}) \oplus \Span{\KRHO{0}}$.

Let $\phi : \QQQ{g} \rightarrow \QQQ{g}'$ be the projection.  The proof will require
seven claims.  It follows the same pattern as Steps 2 and 3 of the proof of Lemma
\ref{lemma:v1injective}.  In Claims \ref{claim:ywelldefined}--\ref{claim:yrelation} and
\ref{claim:zrelation}--\ref{claim:vgenerators}, we will obtain a ``minimal'' size generating set
for $\Span{V_2'}$.  In Claims \ref{claim:yimage} and \ref{claim:v2primespans}, we will show that
$\BBB{g}' = \psi'(\Span{V_2'}) + \Span{\KRHO{0}}$.  A dimension count will then establish
the lemma.

\BeginClaims
\begin{claim}
\label{claim:ywelldefined}
Let $s,s_1,s_2 \in B$ satisfy $s \neq s_1,s_2$ and $i(s,s_1)=i(s,s_2)=0$.
Then for all $v \in \HH_1(\Sigma_g;\Z/L)$ and $e_1,e_2 \in \{-1,1\}$, we have
$\phi(X(v,s,s+e_1 s_1)) = \phi(X(v,s,s+e_2 s_2))$.
\end{claim}
\begin{proof}[{Proof of Claim}]
For $1 \leq i \leq 2$, using the second relation in Lemma \ref{lemma:qrelations}, we get that
$$X(v+e_1 s_1+s,e_2 s_2,s), X(v+e_2 s_2+s,e_1 s_1, s) \in \Span{V_1}.$$
Hence
$$\phi(X(v,e_1 s_1 + e_2 s_2+s,s)) = \phi(X(v,e_1 s_1+s,s) + X(v+e_1 s_1+s,e_2 s_2,s)) = \phi(X(v,s,s+e_1 s_1))$$
and
$$\phi(X(v,e_1 s_1 + e_2 s_2+s,s)) = \phi(X(v,e_2 s_2+s,s) + X(v+e_2 s_2+s,e_1 s_1, s)) = \phi(X(v,s,s+e_2 s_2)).$$
The claim follows.
\end{proof}

In light of Claim \ref{claim:ywelldefined}, we will denote by $Y(v,s)$ the image in $V_2'$ of
$X(v,s,s + e' s')$, where $e' \in \{-1,1\}$ and $s' \in B$ are arbitrary elements such that
$X(v,s,s + e' s') \in V_2$.

\begin{claim}
\label{claim:reparamy}
Consider $Y(v,s) \in V_2'$.  Pick $1 \leq i \leq g$ such that $s \in \{a_i,b_i\}$.
Write $v = v_1 + v_2$ with $v_1 \in \Span{a_i,b_i}$
and $v_2 \in \Span{\{\text{$a_j$, $b_j$ $|$ $j \neq i$}\}}$.  Then $Y(v,s) = Y(v_1,s)$.
\end{claim}
\begin{proof}[{Proof of Claim}]
Consider $s' \in \{\text{$a_j$, $b_j$ $|$ $j \neq i$}\}$.  It is enough to show that $Y(v-s',s) = Y(v,s)$.  Pick
$s'' \in \{\text{$a_j$, $b_j$ $|$ $j \neq i$}\}$ such that $s'' \neq s'$ and $i(s',s'')=0$ (this uses the fact
that $g \geq 3$).  Observe that
$Y(v,s) = \phi(X(v,s,s+s''))$ and $Y(v-s',s) = \phi(X(v-s',s,s+s''))$.  Lemma \ref{lemma:newrelation1} says that
\begin{equation}
\label{eqn:reparamy}
X(v-s',s,s+s'') = X(v,s,s+s'') + X(v-s',s',s+s'') - X(v-s'+s,s',s+s'').
\end{equation}
For $w$ equal to $v-s'$ or $v-s'+s$, we have
$$X(w,s',s+s'') = X(w,s+s'',s') = X(w,s,s') + X(w+s,s'',s') \in \Span{V_1}.$$
Applying $\phi$ to both sides of \eqref{eqn:reparamy}, we thus obtain that $Y(v,s) = Y(v-s',s)$, as desired.
\end{proof}

\begin{claim}
\label{claim:yrelation}
For all $1 \leq i \leq g$ and $v \in \Span{a_i,b_i}$, we have
$$Y(v,a_i)-2Y(v+b_i,a_i)+Y(v+2b_i,a_i) = Y(v,b_i) - 2Y(v+a_i,b_i) + Y(v+2a_i,b_i).$$
\end{claim}
\begin{proof}[{Proof of Claim}]
Pick $1 \leq j < k \leq g$ such that $i \neq j,k$ (this uses the fact that $g \geq 3$).  Lemma \ref{lemma:newrelation2} applied with
$(a_1',b_1',a_2',b_2') = (a_i+a_k,b_i-b_j,a_i+a_j,b_i-b_k)$ says that
\begin{align}
&X(v,a_i+a_k,a_i+a_j) - X(v+b_i-b_j,a_i+a_k,a_i+a_j) - X(v+b_i-b_k,a_i+a_k,a_i+a_j) \label{eqn:xrelation}\\
&+ X(v+b_i-b_j+b_i-b_k,a_i+a_k,a_i+a_j) \notag\\
= &X(v,b_i-b_j,b_i-b_k) - X(v+a_i+a_k,b_i-b_j,b_i-b_k) - X(v+a_i+a_j,b_i-b_j,b_i-b_k) \notag\\
&+ X(v+a_i+a_k+a_i+a_j,b_i-b_j,b_i-b_k). \notag
\end{align}
Since $X(v+a_i,a_k,a_i+a_j) \in \Span{V_1}$, we have that
$$\phi(X(v,a_i+a_k,a_i+a_j)) = \phi(X(v,a_i,a_i+a_j) + X(v+a_i,a_k,a_i+a_j)) = Y(v,a_i).$$
Similarly, we have $\phi(X(v+b_i-b_j,a_i+a_k,a_i+a_j)) = Y(v+b_i-b_j,a_i)$.  By Claim \ref{claim:reparamy},
this equals $Y(v+b_i,a_i)$.  Continuing in this manner, we deduce that $\phi$ maps \eqref{eqn:xrelation}
to the desired relation between the $Y(\cdot,\cdot)$.
\end{proof}

For the next claim, recall that we are using Lemma \ref{lemma:v1injective} to identify
$\BBB{g}' = \BBB{g} / \psi(\Span{V_1})$
with the subspace $\Span{\{\text{$\KRHO{c a_i + d b_i}$ $|$ $c,d \in \Z/L$, $1 \leq i \leq g$}\}}$
of $\BBB{g}$.

\begin{claim}
\label{claim:yimage}
For some $1 \leq i \leq g$, let $s \in \{a_i,b_i\}$ and $v \in \Span{a_i,b_i}$.  Then
$\psi'(Y(v,s)) = \KRHO{v} - 2\KRHO{v+s} + \KRHO{v+2s}$.
\end{claim}
\begin{proof}[{Proof of Claim}]
Let $\rho : \BBB{g} \rightarrow \BBB{g}'$ be the projection.  Pick $1 \leq j \leq g$ such that $j \neq i$.  Observe
that $Y(v,s) = \phi(X(v,s,s-a_j))$ and $X(v+s-a_j,a_j,s) \in V_1$.  Thus $\psi'(Y(v,s))$ equals
\begin{align*}
\rho(\psi(X(v,s-a_j,s))) &= \rho(\psi(X(v,s-a_j,s) + X(v+s-a_j,a_j,s))) \\
&= \rho((\KRHO{v} - \KRHO{v+s-a_j} - \KRHO{v+s} + \KRHO{v+2s-a_j}) \\
&\quad\quad + (\KRHO{v+s-a_j} - \KRHO{v+s} - \KRHO{v+2s-a_j} + \KRHO{v+2s})\\
&= \rho(\KRHO{v} - 2\KRHO{v+s} + \KRHO{v+2s}).
\end{align*}
Since $v,v+s,v+2s \in \Span{a_i,b_i}$, this equals $\KRHO{v} - 2 \KRHO{v+2} + \KRHO{v+2s}$, as desired.
\end{proof}

For some $1 \leq i \leq g$, consider $s \in \{a_i,b_i\}$ and $v \in \Span{a_i,b_i}$.  Making use of Claim
\ref{claim:yimage}, an easy induction establishes that for $n \geq 1$, we have
$$\psi'(\sum_{k=1}^n k \cdot Y(v+(k-1)s,s)) = \KRHO{v} - (n+1) \KRHO{v+n s} + n \KRHO{v+(n+1) s}.$$
In particular, setting $Z(v,s) = \sum_{k=1}^{L} k \cdot Y(v+(k-1)s,s)$, we have
\begin{equation}
\label{eqn:zformula}
\psi'(Z(v,s)) = L \KRHO{v+s} - L \KRHO{v}.
\end{equation}
We now prove the following.

\begin{claim}
\label{claim:zrelation}
For all $1 \leq i \leq g$ and $v \in \Span{a_i,b_i}$, we have
$$Z(v,a_i) - 2 Z(v+b_i,a_i) + Z(v+2 b_i,a_i) = L \cdot Y(v+a_i,b_i) - L \cdot Y(v, b_i).$$
\end{claim}
\begin{proof}[{Proof of Claim}]
Observe that $Z(v,a_i) - 2 Z(v+b_i,a_i) + Z(v+2 b_i,a_i)$ equals
$$\sum_{k=1}^L k \cdot (Y(v+(k-1)a_i,a_i) - 2 Y(v+(k-1)a_i+b_i,a_i) + Y(v+(k-1)a_i+2 b_i,a_i)).$$
By Claim \ref{claim:yrelation}, this equals
$$\sum_{k=1}^L k \cdot (Y(v+(k-1)a_i,b_i) - 2Y(v+k a_i,b_i) + Y(v+(k+1)a_i,b_i)).$$
An argument similar to the argument used to calculate $\psi'$ of $Z(\cdot,\cdot)$ then
shows that this equals $L \cdot Y(v+a_i,b_i) - L \cdot Y(v, b_i)$, and we are done.
\end{proof}

\begin{claim}
\label{claim:vgenerators}
We have
\begin{align*}
\Span{V_2'} = &\langle \{\text{$Y(c a_i+d b_i, a_i)$ $|$ $1 \leq i \leq g$, $c,d \in \Z/L$, $c \neq L-1$}\} \\
&\quad \cup \{\text{$Y(d b_i,b_i)$ $|$ $1 \leq i \leq g$, $d \in \Z/L$, $d \neq L-1$}\} \rangle.
\end{align*}
\end{claim}
\begin{proof}[{Proof of Claim}]
Claim \ref{claim:reparamy} implies that
\begin{equation}
\label{eqn:vgenerators}
\Span{V_2'} = \{\text{$Y(v,s)$ $|$ $v \in \Span{a_i,b_i}$ and $s \in \{a_i,b_i\}$ for some $1 \leq  i \leq g$}\}.
\end{equation}
If $v \in \{a_i,b_i\}$ for some $1 \leq i \leq g$, then
$$Z(v,a_i) \in \langle \{\text{$Y(c a_i+d b_i, a_i)$ $|$ $1 \leq i \leq g$, $c,d \in \Z/L$}\} \rangle.$$
We can therefore use the relation in Claim \ref{claim:zrelation} to reduce \eqref{eqn:vgenerators} to
\begin{align}
\Span{V_2'} = &\langle \{\text{$Y(c a_i+d b_i, a_i)$ $|$ $1 \leq i \leq g$, $c,d \in \Z/L$}\} \label{eqn:bigset}\\
&\quad \cup \{\text{$Y(d b_i,b_i)$ $|$ $1 \leq i \leq g$, $d \in \Z/L$}\} \rangle. \notag
\end{align}
Finally, if $v \in \Span{a_i,b_i}$ and $s \in \{a_i,b_i\}$ for some $1 \leq i \leq g$, then from
the third relation in Lemma \ref{lemma:qrelations}, we obtain the relation
$\sum_{k=0}^{L-1} Y(v+k \cdot s,s) = 0$ in $\QQQ{g}'$.  This allows us to reduce \eqref{eqn:bigset} to
the claimed generating set, and we are done.
\end{proof}

\begin{claim}
\label{claim:v2primespans}
$\BBB{g}' = \Span{\psi'(V_2')} + \Span{\KRHO{0}}$.
\end{claim}
\begin{proof}[{Proof of Claim}]
Consider $c,d \geq 0$ with $(c,d) \neq 0$.  With the convention that an empty sum
of abelian group elements is the zero element, we can use \eqref{eqn:zformula} to get that
\begin{align*}
\frac{1}{L} \psi'(\sum_{j=1}^c Z(j a_i, a_i) + \sum_{k=1}^d Z(c a_i + k b_i,b_i)) &= (\KRHO{c a_i} - \KRHO{0}) + (\KRHO{c a_i+d b_i} - \KRHO{c a_i}) \\
&= \KRHO{c a_i + d b_i} - \KRHO{0}.
\end{align*}
Here the first equality follows from the fact that the indicated sums become telescoping sums after applying
$\psi'$.  The claim follows.
\end{proof}

Observe now that the generating set for $\Span{V_2'}$ given by Claim \ref{claim:vgenerators} has
$$g(L(L-1) + L-1) = g(L^2-1) = \Dim(\BBB{g}') - 1$$
elements, so by Claim \ref{claim:v2primespans} we have that $\psi'|_{\Span{V_2'}}$ is injective
and $\BBB{g}' = \Span{\psi'(V_2')} \oplus \Span{\KRHO{0}}$, as desired.
\end{proof}

\section{Killing off $\SSS$}
\label{section:kills}

This section is devoted to the proof of Lemma \ref{lemma:killsker}.  The proof
itself is contained in \S \ref{section:killsproof}.  This is proceeded by \S \ref{section:killslemma},
which contains a technical lemma about essentially separate curves.

\subsection{Separating essentially separate curves}
\label{section:killslemma}

This section is devoted to the proof of the following lemma.

\begin{lemma}
\label{lemma:separateessentiallyseparate}
Consider $v,v' \in H_L$ and $x,y \in \KKK{g}$ such that $x$ and $y$ are essentially
separate and $\CLL{y} \in \CCC{g}$.  There then exists some $n \geq 1$ and $v_1,v_1',\ldots,v_n,v_n' \in H_L$ and 
$x_1,y_1,\ldots,x_n,y_n \in \KKK{g}$ with the following properties.
\begin{enumerate}
\item $\CL{x}{v} \otimes \CL{y}{v'} = \sum_{i=1}^n \CL{x_i}{v_i} \otimes \CL{y_i}{v_i'}$.
\item For all $1 \leq i \leq n$, one of the following two conditions is satisfied.
\begin{enumerate}
\item $\{x_i,y_i\}$ has the same unoriented intersection pattern 
as the curves in Figure \ref{figure:separate}.a, or
\item $x_i = (x'_i)^L$, where $\{x_i',y_i\}$ has the same unoriented 
intersection pattern as the curves in Figure \ref{figure:separate}.b.
\end{enumerate}
\end{enumerate}
\end{lemma}

\Figure{figure:separate}{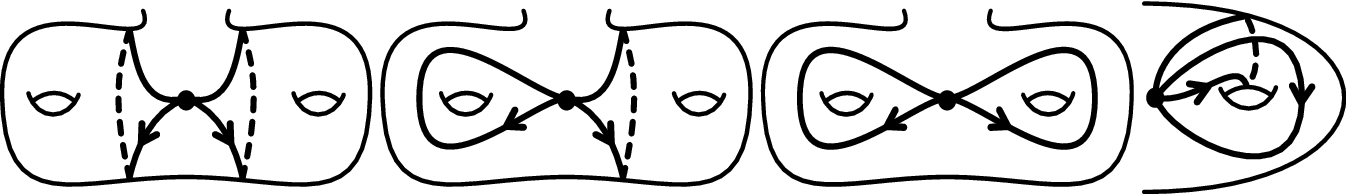}{a,b. The intersection patterns needed for Lemma 
\ref{lemma:separateessentiallyseparate}
\CaptionSpace c. A nonseparating figure eight.
\CaptionSpace d. $[a,b]$ is a genus $1$ separating curve.}

\begin{remark}
The second condition in the conclusion of Lemma \ref{lemma:separateessentiallyseparate} implies
that $\CLL{y_i} \in \CCC{g}$.
\end{remark}

To prove Lemma \ref{lemma:separateessentiallyseparate}, we will need a definition and a lemma.

\begin{definition}
A curve $x \in \pi_1(\Sigma_g)$ is a {\em genus $k$ separating curve} if it can be realized
by a simple closed curve that separates $\Sigma_g$ into two pieces, one of which is homeomorphic
to $\Sigma_{k,1}$.  Two curves $y,z \subset \pi_1(\Sigma_g)$ form a {\em nonseparating figure eight} if
they have the same unoriented intersection pattern as the curves in Figure \ref{figure:separate}.c.
\end{definition}

\begin{remark}
If $a,b \in \pi_1(\Sigma_g)$ have the same unoriented intersection pattern as the curves
in Figure \ref{figure:separate}.d, then $[a,b] \in \pi_1(\Sigma_g)$ is a genus $1$ separating
curve.
\end{remark}

\begin{lemma}
\label{lemma:basis}
Let $S \subset \Sigma_g$ be a subsurface such that the basepoint lies on $\partial S$ and
$\Sigma_g \setminus S$ is connected.  Define $P_S = \Image(\pi_1(S) \rightarrow \pi_1(\Sigma_g))$
and $H_S = \Image(\HH_1(S;\Z/L) \rightarrow H_L)$.
Next, define the following sets.
\begin{align*}
U_1 &= \{\text{$\CLL{x^L}$ $|$ $x \in P_S$ is a simple closed nonseparating curve}\},\\
U_2 &= \{\text{$\CL{x}{v}$ $|$ $v \in H_S$, $x \in P_S$ is a genus $1$ separating curve}\},\\
U_3 &= \{\text{$\CL{y,z}{v}$ $|$ $v \in H_S$, $y,z \subset P_S$ form a nonseparating figure eight}\}.
\end{align*}
Then the following hold.
\begin{enumerate}
\item $\CCC{g} \cap \HH_1(\KKK{g} \cap P_S;\Q)$ is spanned by $U_2 \cup U_3$ and
$\HH_1(\KKK{g} \cap P_S;\Q)$ is spanned by $U_1 \cup U_2 \cup U_3$
\item If the genus of $S$ is positive, then $\CCC{g} \cap \HH_1(\KKK{g} \cap P_S;\Q)$ is spanned by $U_2$ and
$\HH_1(\KKK{g} \cap P_S;\Q)$ is spanned by $U_1 \cup U_2$.
\end{enumerate}
\end{lemma}
\begin{proof}
Define
$$\KKK{S} = \Ker(\pi_1(S) \rightarrow \HH_1(S;\Z/L)) \quad \quad \text{and} \quad \quad \CCC{S} = \Ker(\HH_1(\KKK{S};\Q) \rightarrow \HH_1(S;\Q)).$$
Since $\Sigma_g \setminus S$ is connected, the map $\HH_1(S;\Z/L) \rightarrow \HH_1(\Sigma_g;\Z/L)$ is injective.
Using the commutative diagram
$$\begin{CD}
\pi_1(S) @>>> \pi_1(\Sigma_g)\\
@VVV          @VVV \\
\HH_1(S;\Z/L) @>>> \HH_1(\Sigma_g;\Z/L)
\end{CD}$$
we deduce that the natural maps
$$\KKK{S} \longrightarrow \KKK{g} \cap P_S \quad \quad \text{and} \quad \quad \CCC{S} \longrightarrow \CCC{g} \cap \HH_1(\KKK{g} \cap P_S;\Q)$$
are surjective.  

Define
\begin{align*}
V_1' &= \Set[,]{$x^L \in \KKK{S}$}{$x \in \pi_1(S)$ can be realized by a simple closed curve that is either\\
nonseparating or freely homotopic to a boundary component}\\
V_2' &= \Set[,]{$x^w \in \KKK{S}$}{$x,w \in \pi_1(S)$ and $x$ maps to a genus $1$ separating curve in $\pi_1(\Sigma_g)$}\\
V_3' &= \Set[.]{$[y,z]^w \in \KKK{S}$}{$y,z,w \in \pi_1(S)$ and $\{y,z\}$ maps to a nonseparating figure\\
eight in $\pi_1(\Sigma_g)$}
\end{align*}
There is a short exact sequence
$$1 \longrightarrow [\pi_1(S),\pi_1(S)] \longrightarrow \KKK{S} \longrightarrow L \cdot \HH_1(S;\Z) \longrightarrow 1,$$
where $L \cdot \HH_1(S;\Z)$ denotes the subgroup $\{\text{$L \cdot v$ $|$ $v \in \HH_1(S;\Z)$}\}$.  The subgroup
of $\KKK{S}$ generated by $V_1'$ surjects onto $L \cdot \HH_1(S;\Z)$.  Also, making use
of a standard basis for $\pi_1(S)$, one can easily check that $V_2' \cup V_3'$ generates $[\pi_1(S),\pi_1(S)]$.
Finally, the proof of \cite[Theorem A.3]{PutmanCutPaste} shows that if the genus of $S$ is positive, then
$[\pi_1(S),\pi_1(S)]$ is generated by $V_2'$.  The upshot of this is that $\KKK{S}$ is generated
by $V_1' \cup V_2' \cup V_3'$ in all cases and by $V_1' \cup V_2'$ if the genus of $S$ is positive.

Defining $V_i$ to be the image of $V_i'$ in $\HH_1(\KKK{S};\Q)$, we obtain that
$\HH_1(\KKK{S};\Q)$ (resp. $\CCC{S}$) is spanned by $V_1 \cup V_2 \cup V_3$ (resp.\ $V_2 \cup V_3$) 
in all cases and by $V_1 \cup V_2$ (resp.\ $V_2$) if the genus of $S$ is positive.
The set $V_i$ maps to $U_i$ under the natural map $\HH_1(\KKK{S};\Q) \rightarrow \HH_1(\KKK{g} \cap P_S;\Q)$.  
Since this map and the restricted map $\CCC{S} \rightarrow \CCC{g} \cap \HH_1(\KKK{g} \cap P_S;\Q)$ 
are surjections, the lemma follows
\end{proof}

\begin{proof}[{Proof of Lemma \ref{lemma:separateessentiallyseparate}}]
Since $x$ and $y$ are essentially separate, we can decompose $\Sigma_g$ into the union
of two connected subsurfaces $S_1$ and $S_2$ with the following properties.
\begin{enumerate}
\item $S_1$ and $S_2$ both contain the basepoint.
\item $S_1 \cap S_2 = \partial S_1 = \partial S_2$.
\item $x \in \Image(\pi_1(S_1) \rightarrow \pi_1(\Sigma_g))$ and $y \in \Image(\pi_1(S_2) \rightarrow \pi_1(\Sigma_g))$.
\end{enumerate}
Applying Lemma \ref{lemma:basis} to each $S_i$, we can write
$$\CLL{x} = \sum_{i=1}^{k} \CL{x_i}{v_i} \quad \quad \text{and} \quad \quad \CLL{y} = \sum_{j=1}^{k'} \CL{y_j}{v_j'},$$
where $v_i,v_j' \in H_L$ and $x_i$ and $y_j$ satisfy the following conditions.
\begin{itemize}
\item $x_i \in \Image(\pi_1(S_1) \rightarrow \pi_1(\Sigma_g))$ and 
$y_j \in \Image(\pi_1(S_2) \rightarrow \pi_1(\Sigma_g))$ for all $i$ and $j$.  In particular, the curves
$x_i$ and $y_j$ are essentially separate for all $i$ and $j$.
\item $x_i$ is either a genus $1$ separating curve, the commutator of a nonseparating figure eight, or $z^L$ for some simple
closed nonseparating curve $z$.
\item $y_j$ is either a genus $1$ separating curve or the commutator of a nonseparating figure eight.
\end{itemize}
We then have
$$\CL{x}{v} \otimes \CL{y}{v'} = \sum_{i=1}^{k} \sum_{j=1}^{k'} \CL{x_i}{v+v_i} \otimes \CL{y_j}{v'+v_j'}.$$
We are almost done -- the only problem is that $x_i$ or $y_j$ might be the commutator of 
a nonseparating figure eight for some $i$ or $j$.  However, if that happens, then we may perform 
the above procedure again to the pair $\{x_i,y_j\}$,
but this time it is easy to see that we may ensure that the genera of $S_1$ and $S_2$ are both positive, in
which case we do not need to use nonseparating figure eights.
\end{proof}

\subsection{The proof of Lemma \ref{lemma:killsker}}
\label{section:killsproof}

In this final section, we prove Lemma \ref{lemma:killsker}.  We begin by
briefly recalling its statement.  The set
$$\SKER < \HH_1(\KKK{g};\Z) \otimes \CCC{g} = \HH_1(\KKK{g};\CCC{g})$$
is the span of the set
$$\{\text{$x \otimes y$ $|$ $(x,y) \in \SSS$}\} \cup \{\text{$x \otimes y - f(x) \otimes f(y)$ $|$ $x \in \HH_1(\KKK{g};\Z)$, $y \in \CCC{g}$, $f \in \Mod_g^1(L)$}\}.$$
Here $\SSS$ is the set defined in \S \ref{section:mainlemmaproof}.  Lemma
\ref{lemma:killsker} asserts that for $g \geq 4$, the image of $\SKER$ in
$\HH_1(\Mod_g^1(L);\CCC{g})$ is zero.

Let $\phi : \HH_1(\KKK{g};\Z) \otimes \CCC{g} \rightarrow \HH_1(\Mod_g^1(L);\CCC{g})$ be the
natural map.  Since inner automorphisms act trivially on homology, we have
\begin{equation}
\label{eqn:modsymmetry}
\phi(x \otimes y) = \phi(f(x) \otimes f(y))
\end{equation}
for all $x \otimes y \in \HH_1(\KKK{g};\Z) \otimes \CCC{g}$ and $f \in \Mod_g^1(L)$.  We
must show that in addition we have $\phi(s) = 0$ for $s \in \SSS$.  This will require
several steps.

\BeginSteps
\begin{step}
Consider $v,v' \in H_L$ and $x,y \in \KKK{g}$ such that $x$ and $y$ are essentially
separate and $\CLL{y} \in \CCC{g}$, so 
$(\CL{x}{v}, \CL{y}{v'}) \in \SSSS{1}$.  Then $\phi(\CL{x}{v} \otimes \CL{y}{v'})=0$.
\end{step}

By Lemma \ref{lemma:separateessentiallyseparate}, we can assume that one of the following
two conditions hold.
\begin{enumerate}
\item $\{x,y\}$ has the same unoriented intersection pattern as the curves in Figure \ref{figure:separate}.a, or
\item $x = (x')^L$, where $\{x',y\}$ has the same unoriented intersection pattern as
the curves in Figure \ref{figure:separate}.b.
\end{enumerate}
Thus $y$ is a separating curve separating $\Sigma_g$ into two subsurfaces $X_1$
and $X_2$ with $X_2 \cong \Sigma_{1,1}$ (see Figure \ref{figure:kills}.a).  Also,
$x$ is in $\Image(\pi_1(X_1) \rightarrow \pi_1(\Sigma_g))$.

\Figure{figure:kills}{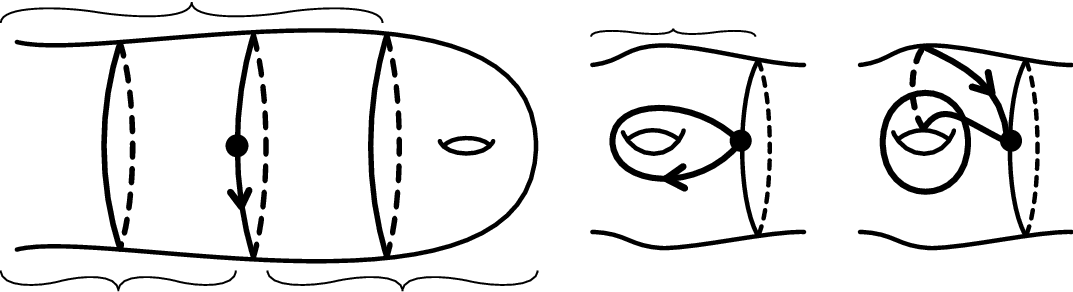}{a. Figure needed to deal with elements of $\SSSS{1}$.
\CaptionSpace b. The curve $z$ can be realized by a simple closed nonseparating curve and
lies entirely in $X \subset \Sigma_g$.
\CaptionSpace c. $T_{\delta}^{e L}(w) = z^L w$.  The sign $e=\pm 1$ depends on the orientation
of $z$; with the orientation depicted in the figure, we have $e=1$.}

The group $\Mod_g^1(L)$ contains all inner automorphisms of $\pi_1(\Sigma_g)$, so
using \eqref{eqn:modsymmetry} we can assume that $v=0$.  We will produce
a subgroup $\Gamma < \Mod_g^1(L)$ with the following three properties.
\begin{enumerate}
\item $\HH_1(\Gamma;\Q) = 0$.
\item $\Gamma$ acts trivially on $\CLL{y}$.
\item $\Gamma$ contains the ``point-pushing'' mapping class corresponding to $x \in \pi_1(\Sigma_g)$.
\end{enumerate}
This is enough to prove the desired claim.  Indeed, since $\Gamma < \Mod_g^1(L)$, we have
$f(v')=v'$ for all $f \in \Gamma$.  This and item 2 imply that $\Gamma$ fixes $\CL{y}{v'}$.  There
is thus a map $j:\HH_1(\Gamma;\Q) \rightarrow \HH_1(\Mod_g^1(L);\CCC{g})$ corresponding
to the inclusion $\Gamma \hookrightarrow \Mod_g^1(L)$ and the map $\Q \hookrightarrow \CCC{g}$ taking
$1 \in \Q$ to $\CL{y}{v'}$.  By item 3, the image of $j$ contains $\phi(\CL{x}{v} \otimes \CL{y}{v'})$,
so by item 1 we have $\phi(\CL{x}{v} \otimes \CL{y}{v'})=0$, as desired.

We construct $\Gamma$ as follows.  Let $N$ be a regular neighborhood of $y$ and let $\delta_i$
be the boundary component of $N$ that is contained in $X_i$ (see Figure \ref{figure:kills}.a).
The curve $\delta_i$ separates $\Sigma_g$ into two components.  Let $Y_i$ be the component
of $\Sigma_g$ cut along $\delta_i$ that intersects $X_1$.  It follows that $Y_1 \cong \Sigma_{g-1,1}$
and $Y_2 \cong \Sigma_{g-1,1}^1$.  Denoting the mapping class group of $Y_i$ by $\Mod(Y_i)$, we 
have a Birman exact sequence
$$1 \longrightarrow \pi_1(\Sigma_{g-1,1}) \longrightarrow \Mod(Y_2) \longrightarrow \Mod_{g-1,1} \longrightarrow 1.$$
We have an isomorphism $\Mod_{g-1,1} \cong \Mod(Y_1)$, and this exact sequence splits via a map
that identifies $f \in \Mod_{g-1,1}$ with a corresponding element of $\Mod(Y_1)$ and then
extends $f$ by the identity to an element of $\Mod(Y_2)$.  Choosing such a splitting, we get
a decomposition
$$\Mod(Y_2) = \pi_1(\Sigma_{g-1,1}) \rtimes \Mod_{g-1,1}.$$
Define
$$\Gamma = \KKK{g-1,1} \rtimes \Mod_{g-1,1}(L) < \pi_1(\Sigma_{g-1,1}) \rtimes \Mod_{g-1,1}.$$
It is clear that $\Gamma$ contains the ``point-pushing'' mapping class corresponding
to $x \in \pi_1(\Sigma_g)$.  Also, the subgroup $\Mod_{g-1,1}(L) < \Gamma$ fixes $y \in \pi_1(\Sigma_g)$
and the subgroup $\KKK{g-1,1} < \Gamma$ acts on $y$ by conjugation, and thus fixes $\CLL{y}$.  We
deduce that $\Gamma$ acts trivially on $\CLL{y}$.

It remains to check that $\HH_1(\Gamma;\Q)=0$.  From its semidirect product decomposition,
we deduce that
$$\HH_1(\Gamma;\Q) \cong \HH_1(\Mod_{g-1,1}(L);\Q) \oplus (\HH_1(\KKK{g-1,1};\Q))_{\Mod_{g-1,1}(L)}.$$
Since $g \geq 4$, Theorem \ref{theorem:h1modl} implies that $\HH_1(\Mod_{g-1,1}(L);\Q) = 0$.  Again
using the fact that $g \geq 4$,
Lemma \ref{lemma:killcg} implies that $(\HH_1(\KKK{g-1,1};\Q))_{\Mod_{g-1,1}(L)}=0$, and we are done.

\begin{step}
$\phi(a \otimes b)=0$ if $(a,b) \in \SSSS{2}$.
\end{step}
Consider $(\CL{x}{v}, \CL{y,z^L}{v'}) \in \SSSS{2}$, so $z$ can be realized
by a simple closed nonseparating curve and $\{z\}$ and $\{x,y\}$ are strongly
essentially disjoint.  We can thus find subsurfaces $X$ and $X'$ of $\Sigma_g$
both of which contain the basepoint such that
$$z \in \Image(\pi_1(X) \rightarrow \pi_1(\Sigma_g)) \quad \quad \text{and} \quad \quad
\{x,y\} \subset \Image(\pi_1(X') \rightarrow \pi_1(\Sigma_g)).$$
and such that
$$\Sigma_g = X \cup X' \quad \quad \text{and} \quad \quad X \cap X' = \partial X = \partial X'.$$
Additionally, the surfaces $X$ and $X'$ can be chosen such that both have only one
boundary component.

As is shown in Figure \ref{figure:kills}.c, there exists
some $w \in \Image(\pi_1(X) \rightarrow \pi_1(\Sigma_g))$ together with an unbased simple closed curve $\delta$ 
in $X \subset \Sigma_g$ such that $T_{\delta}^{e L}(w) = z^L w$ for some $e = \pm 1$.  
Since $\delta \subset X$, we have
$T_{\delta}(x) = x$ and $T_{\delta}(y) = z$.  Also, since $T_{\delta}^{e L} \in \Mod_{g}^1(L)$,
it follows that $T_{\delta}^{e L}$ acts trivially on $v$ and $v'$.  This implies
that $T_{\delta}^{e L}$ takes $\CL{x}{v} \otimes \CL{y,w}{v'}$ to
\begin{align*}
\CL{x}{v} \otimes \CL{y,z^L w}{v'} &= \CL{x}{v} \otimes \CL{y,z^L}{v'} + \CL{x}{v} \otimes \CL{y,w}{v'+ L \cdot \CAB{z}} \\
&= \CL{x}{v} \otimes \CL{y,z^L}{v'} + \CL{x}{v} \otimes \CL{y,w}{v'}.
\end{align*}
The first calculation here uses Lemma \ref{lemma:commutatoridentities} and the second the fact that
$L \cdot \CAB{z} = 0$.  Since $\Mod_{g}^1(L)$ acts trivially on the image of $\phi$, we conclude that
$$\phi(\CL{x}{v} \otimes \CL{y,w}{v'}) = \phi(\CL{x}{v} \otimes \CL{y,w}{v'} + \CL{x}{v} \otimes \CL{y,z^L}{v'});$$
i.e.\ that $\phi(\CL{x}{v} \otimes \CL{y,z^L}{v'}) = 0$, as desired.

\noindent
{\raggedright
Andrew Putman\\
Department of Mathematics\\
Rice University, MS 136 \\
6100 Main St.\\
Houston, TX 77005\\ 
E-mail: {\tt andyp@rice.edu}}

\end{document}